\documentclass[a4paper]{amsart}

\usepackage[all]{xy}
\usepackage{amsthm,amssymb}
\usepackage{latexsym}
\usepackage{amsmath}
\usepackage{extarrow}
\usepackage{enumerate}
\usepackage{graphicx}
\usepackage{tikz}
\usepackage[colorlinks=true, linkcolor=blue]{hyperref}
\usetikzlibrary{arrows,matrix,shapes,snakes}
\usepackage{upgreek}
\usepackage{textcomp}
\usepackage{mathrsfs}
\usepackage{fontenc,mathtext}
\usepackage{mathrsfs}
\usepackage{cite}

\newtheorem{cont}{cont}[subsection]
\newtheorem{thm}[cont]{Theorem}
\newtheorem*{thm*}{Theorem}
\newtheorem{prop}[cont]{Proposition}
\newtheorem{lemma}[cont]{Lemma}
\newtheorem{cor}[cont]{Corollary}
\newtheorem{dfn}[cont]{Definition}
\numberwithin{equation}{section}

\theoremstyle{remark}
\newtheorem{rem}[cont]{Remark}

\newcommand{\shF}{\mathcal{F}}

\newcommand{\shG}{\mathcal{G}}

\newcommand{\shH}{\mathcal{H}}
\newcommand{\shI}{\mathcal{I}}

\newcommand{\shL}{\mathcal{L}}
\newcommand{\shM}{\mathcal{M}}
\newcommand{\shE}{\mathcal{E}}

\newcommand{\OO}{\mathcal{O}}
\newcommand{\odi}[1]{\mathcal{O}_{#1}}
\newcommand{\then}{\ \Longrightarrow \ }
\newcommand{\arr}{\longrightarrow}
\newcommand{\arrdi}[1]{\xlongrightarrow{#1}}
\newcommand{\E}{\mathcal{E}}
\newcommand{\duale}[1]{{#1}^{\vee}}
\newcommand{\PP}{\mathbb{P}}

\DeclareMathOperator{\Hl}{H}
\DeclareMathOperator{\h}{h}
\DeclareMathOperator{\Pic}{Pic}

\DeclareMathOperator{\rk}{rk}
\DeclareMathOperator{\Hom}{Hom}
\newcommand{\Ext}{\textup{Ext}}

\DeclareMathOperator{\depth}{depth}

\DeclareMathOperator{\di}{dim}
\DeclareMathOperator{\id}{id}

\DeclareMathOperator{\de}{deg}

\begin{document}

\title{ACM BUNDLES ON DEL PEZZO SURFACES}

\author[Joan Pons-Llopis]{Joan Pons-Llopis}

\address{Facultat de Matem\`{a}tiques, Universitat de Barcelona, Gran V\'ia 585, 08007 Barcelona,Spain.}

\email{jfpons@ub.edu}

\author[Fabio Tonini]{Fabio Tonini}

\address{Scuola Normale Superiore, Piazza dei Cavalieri 7, 56100, Pisa, Italy}

\email{fabio.tonini@sns.it}

\date{}

\thanks{The first author is supported by the research project MTM2006-04785.}


\maketitle

\begin{abstract}
\noindent ACM rank 1 bundles on del Pezzo surfaces are classified in terms of the rational normal curves that they contain. A complete list
of ACM line bundles is provided. Moreover, for any del Pezzo surface $X$ of degree less or equal than six and for any $n\geq 2$ we construct a family of dimension $\geq n-1$ of non-isomorphic simple ACM bundles of rank $n$ on $X$.
\end{abstract}

\section{Introduction}
Given a $n$-dimensional smooth projective variety $X$ over an algebraically closed field $k$ and a very ample line bundle $\OO_X(1)$ on it, associated to any vector bundle $\E$ on $X$ we have the cohomology groups $\Hl^i(X,\E(l))$ for $i\in\{0,\ldots ,n\}$ and $l\in\mathbb{Z}$. It's well-known that for $l$ big enough $\Hl^0(X,\E(l))\neq 0$ and by Serre duality there also exists $l$ such that $\Hl^n(X,\E(l))\neq 0$. Therefore we have only freedom to ask for the vanishing of the intermediate cohomology groups. The vector bundles for which this vanishing is achieved are called \emph{Arithmetically Cohen-Macaulay bundles} (\emph{ACM} for short).

It becomes a natural question to study the complexity of the structure of ACM bundles on a given variety. The first result addressing this question was Horrocks' theorem which states that on $\PP_k^n$ the only indecomposable ACM bundle up to twist is the structure sheaf $\OO_{\PP_k^n}$. Later on, Kn\"{o}rrer in \cite{Kno} proved that on a smooth quadric hypersurface X the only indecomposable ACM bundles up to twist are $\OO_X$ and the spinor bundles $S$ (which are one or two according to the parity of the dimension of the quadric).

A complete list of varieties that admit only a finite number of indecomposable ACM bundles (up to twist and isomorphism) was given in \cite{BGS} and \cite{EH}: assuming that $X$ has only finitely many indecomposable vector bundles, then $X$ is either a projective space $\PP_k^n$, a smooth quadric, a cubic scroll in $\PP_k^4$, the Veronese surface in $\PP_k^5$ or a rational normal curve. They have been called varieties of \emph{finite representation type} (see \cite{DG} and references herein).

On the other extreme there would lie those varieties of \emph{wild representation type}, namely, varieties for which there exist $n$-dimensional families of non-isomorphic indecomposable ACM bundles for arbitrary large $n$. In the one dimensional case, it's known that curves of wild representation type are exactly those of genus larger or equal than two. For varieties of larger dimension, in \cite{CH} Casanellas and Hartshorne were able to construct on a smooth cubic surface in $\PP_k^3$ for any $n\geq 2$ a $n^2+1$-dimensional family of rank $n$ indecomposable ACM vector bundles with Chern classes $c_1=nH$ and $c_2= \frac{1}{2}(3n^2-n)$. Moreover, Faenzi in \cite{Fa} was able to give a precise classification of rank 2 ACM bundles on cubic surfaces. He proved that they fall in 12 classes according to their minimal free resolution as coherent $\OO_{\PP_k^3}$-sheaves.

In this paper we focus our attention on ACM bundles on a class of surfaces that contains the smooth cubic surfaces as a particular case. This class of surfaces, known as \emph{del Pezzo surfaces}, has a very nice description in terms of blow-ups of general points in $\PP_k^2$ and has been broadly studied. Good sources are \cite{Dem} and \cite{Man} where arithmetic aspects of these surfaces are also studied. For a pure geometrical introduction we recommend \cite{Dol}. The first question that we address in this paper is the geometrical characterization of ACM line bundles $\shL$. Since we're interested in bundles up to twist we're only going to work with \emph{initialized} bundles, meaning that $\Hl^0(X,\shL)\neq 0$ but $\Hl^0(X,\shL(-1))=0$. Our result concerning these issues can be stated as follows (see Theorem \ref{maintheorem}):

\begin{thm*}
Let $X\subseteq\PP_k^d$ be a del Pezzo surface of degree $d$ embedded through the very ample divisor $-K_X$. Then a line bundle $\shL$ on $X$ is initialized and ACM if and only if either $\shL\cong\OO_X$ or $\shL\cong\OO_X(D)$ for a rational normal curve $D\subseteq X\subseteq\PP_k^d$ of degree less or equal than $d$.
\end{thm*}

This result was already known in the case of the cubic surface (cfr. \cite{Fa}) and in the case of the del Pezzo surface of degree $4$ (cfr. \cite{Mir}).

Next, we turn our attention to the construction of indecomposable ACM bundles of higher rank for which we use a well-known method: extension of bundles. Thanks to the iteration of this method we obtained the main contribution of this paper, namely, del Pezzo surfaces of degree up to six are of wild representation type by constructing explicitly families of ACM sheaves with the required properties (see Theorem \ref{thetheorem}):

\begin{thm*}
Let $X$ be a del Pezzo surface of degree $\leq 6$. Then for any integer $n \geq 2$ there exists a family of dimension $\geq n-1$ of non-isomorphic simple ACM vector bundles of rank $n$.
\end{thm*}

Let's recall briefly how this paper is organized: in the second section we introduce the necessary background on ACM bundles and del Pezzo surfaces. In section three we stress the properties of the lines that are contained in del Pezzo surfaces and we develop a Bertini-like theorem that expresses which linear systems contain smooth curves in terms of the intersection product with exceptional divisors. Most of the material from this section should be well-known, but we gather it here for the reader's convenience. In section four we classify ACM line bundles on del Pezzo surfaces and give a numerical characterization. In the last section we work out the construction of simple ACM bundles of higher rank.

This paper grows out of the problem that was posed to the authors during the P.R.A.G.MAT.I.C school held at the University of Catania in September 2009. This problem was proposed just in zero characteristic.

\section{Preliminaries}
We follow notation from \cite{Harti}. We are going to work with integral (i.e., reduced and irreducible) varieties over an algebraically closed field $k$ (of arbitrary characteristic). Given a smooth variety $X$ equipped with a very ample line bundle $\odi{X}(1)$ that provides a closed embedding in some $\PP^n_k$, the line bundle $\odi{X}(1)^{\otimes l}$ will be denoted by $\odi{X}(l)$ or $\odi{X}(l H)$. For any coherent sheaf $\E$ on $X$ we're going to denote the twisted sheaf $\shE\otimes\odi{X}(l)$ by $\shE(l)$. As usual, $\Hl^i(X,\E)$ (or simply $\Hl^i(\E)$) stands for the cohomology groups and $\h^i(X,\E)$ (or simply $\h^i(\E)$) for their dimension. For a divisor $D$ on $X$, $\Hl^i(D)$ and $\h^i(D)$ abbreviate $\Hl^i(X,\odi{X}(D))$ and $\h^i(X,\odi{X}(D))$ respectively. We will use the notation $\Hl^i_*(\shE)$ for the graded $k[X_0,\ldots ,X_n]$-module $\bigoplus_{l\in\mathbb{Z}} \Hl^i (\PP_k^n,\shE (l))$. $K_X$ will stand for the canonical class of $X$ and $\omega_X:=\odi{X}(K_X)$ for the canonical bundle.

We're going to say that $\shE$ is \emph{initialized} (with respect to $\odi{X}(1)$) if
$$
\Hl^0(X,\shE(-1))=0 \ \ \text{ but } \  \Hl^0(X,\shE)\neq 0.
$$
If $Y\subseteq X$ is a subvariety we denote the ideal sheaf of $Y$ in $X$ by $\shI_{Y|X}$ and the saturated ideal by $I_{Y|X}:=\Hl^0_*(X,\shI_{Y|X})$. Whenever we write a closed subvariety $X\subseteq\PP^n_k$, we consider it equipped with the very ample line bundle $\odi{\PP_k^n}(1)_{|X}$. We denote by $S_X$ the homogeneous coordinate ring, defined as $k[X_0,\ldots ,X_n]/I_X$.

\subsection{ACM varieties and sheaves}
This subsection will be devoted to recall the definitions and main properties of ACM varieties and sheaves.
\begin{dfn}(cfr. \cite[Chapter I, Definition 1.2.2]{Mig}).
A closed subvariety $X\subseteq\PP^n_k$ is \emph{Arithmetically Cohen-Macaulay (ACM)} if its homogeneous coordinate ring $S_X$ is Cohen-Macaulay or, equivalently, $\di S_X=\depth \ S_X$.
\end{dfn}

Notice that any zero-dimensional variety is ACM. For varieties of higher dimension we have the following characterization that will be used in this paper:
\begin{lemma}(cfr. \cite[ Chapter I, Lemma 1.2.3]{Mig}).\label{ACMcharacterization}
If $\di \ X\geq 1$, then $X\subseteq\PP^n_k$ is ACM if and only if $\Hl^i_*(\shI_X)=0$ for $1\leq i\leq \di X$.
\end{lemma}

\begin{dfn}
Let $X$ be an ACM variety. A coherent sheaf $\E$ on $X$ is \emph{Arithmetically Cohen Macaulay} (ACM for short) if it is locally Cohen-Macaulay (i.e., $\depth \shE_x=\di \odi{X,x}$ for every point $x\in X$) and has no intermediate cohomology:
$$
\Hl^i_*(X,\shE)=0 \quad\quad \text{    for all $i=1, \ldots , \di X-1.$}
$$
\end{dfn}
Notice that when $X$ is a smooth variety, which is going to be mainly our case, any coherent ACM sheaf on $X$ is locally free; for this reason we're going to speak uniquely of ACM bundles.

\begin{lemma}
Let $X\subseteq\PP^n_k$ be an ACM variety. Then $\odi{X}$ is an ACM sheaf (seen as an $\odi{X}$-sheaf).
\end{lemma}
\begin{proof}
The vanishing of $\Hl^i_*(\odi{X})$ is immediate from Lemma \ref{ACMcharacterization} and the short exact sequence defining $X$. On the other hand,
it's a well-known fact that $S_X$ being Cohen-Macaulay implies that $\odi{X,x}$ is Cohen-Macaulay for any $x\in X$.
\end{proof}

Once we work inside an ACM variety, the relation between ACM ideal sheaves and ACM subvarieties is very close:

\begin{lemma}\label{ACMvar}
	Let $X \subseteq \PP^n_k$ be an ACM smooth variety with $\di X \geq 1$ and $D$ be an integral effective divisor on $X$. Then the coherent $\odi{X}$-sheaf $\odi{X}(-D)$ is ACM if and only if  $D \subseteq \PP^n_k$ is an ACM variety.

\end{lemma}

\begin{proof}
Let's consider the exact sequence of $\odi{\PP_k^n}$-sheaves

$$
0\arr \shI_{X| \PP_k^n} \arr\shI_{D| \PP_k^n} \arr\odi{X}(-D) \arr 0.
$$	

If we tensor it with $\odi{\PP_k^n}(t)$ and take cohomology we get

$$
\Hl^i(\shI_{X| \PP_k^n}(t)) \arr \Hl^i(\shI_{D| \PP_k^n}(t)) \arr \Hl^i(\odi{X}(-D)(t)) \arr \Hl^{i+1}(\shI_{X| \PP_k^n}(t)).
$$	

Since $X$ is ACM, both extremes are zero for any $t$ and for $1\leq i\leq \di(X)-1$. Therefore we get isomorphisms

$$
\Hl^i(\shI_{D| \PP_k^n}(t)) \cong \Hl^i(\odi{X}(-D)(t))
$$
for any $t$ and for $1\leq i\leq \di D=\di(X)-1$. Since $\odi{X,x}(-D)\cong\odi{X,x}$ is Cohen-Macaulay for any $x\in X$, this turns out to be enough to conclude.

\end{proof}

\subsection{Del Pezzo surfaces}

In this paper we're going to be interested in ACM bundles on del Pezzo surfaces. This kind of surfaces were studied by P. del Pezzo in the nineteenth century and ever since its presence has been pervasive in Algebraic Geometry. Let's recall their definition and main properties:

\begin{dfn}(cfr. \cite[Chapter III, Definition 3.1]{Kol}).
A \emph{del Pezzo surface} is defined to be a smooth surface $X$ whose anticanonical divisor $-K_X$ is
ample. Its degree is defined as $K_X^2$.
\end{dfn}

\begin{rem}
It's possible to see that del Pezzo surfaces are rational. Indeed, according to Castelnuovo's criterion (cfr. \cite[Chapter III, Theorem 2.4]{Kol}), a smooth surface $X$ is rational if and only if $\h^0(\odi{X}(2K_X))=0$ and $\h^1(\odi{X})=0$. In the case of del Pezzo surfaces, the former cohomology group is zero because $-2K_X$ is ample and therefore clearly $2K_X$ is not effective. In characteristic zero the latter cohomology group is zero thanks to the Kodaira vanishing theorem. In characteristic positive, the vanishing still holds (cfr. \cite[Chapter III, Lemma 3.2.1]{Kol}).
\end{rem}

\begin{rem}[Serre duality for del Pezzo surfaces]
Let $X$ be a del Pezzo surface with very ample anticanonical divisor $H_X:=-K_X$. Given a locally free sheaf $\shE$ Serre duality takes the form:
$$
\Hl^i(X,\E)\cong \Hl^{2-i}(X,\shE^{\vee}(-H_X))'.
$$
This remark will be used without further mention throughout the paper.
\end{rem}

\begin{dfn}
Given a surface $X$, a curve $C$ on $X$ is called exceptional if $C\cong\PP_k^1$ and the self-intersection $C^2=-1$.
\end{dfn}

\begin{thm}(cfr. \cite[Chapter IV,Theorem 24.3]{Man}).
Let $X$ be a del Pezzo surface of degree $d$. Then every irreducible curve with a negative self-intersection number is exceptional.

\end{thm}

\begin{dfn}
A set of points $\{p_1,\ldots ,p_r\}$ on $\PP_k^2$ with $r\leq 9$ are in \emph{general position} if no three of them lie on a line
and no six of them lie on a conic.
\end{dfn}

The following theorem characterizes all del Pezzo surfaces:

\begin{thm}(cfr. \cite[Chapter IV, Theorems 24.3 and 24.4]{Man}).
Let $X$ be a del Pezzo surface of degree $d$. Then $1\leq d\leq 9$ and
\begin{enumerate}
\item[(i)] If $d=9$, then $X$ is isomorphic to $\PP_k^2$ (and $-K_{\PP^2_k}=3H_{\PP^2_k}$ gives the usual Veronese embedding in $\PP_k^9$).
\item[(ii)] If $d=8$, then $X$ is isomorphic to either $\PP_k^1\times\PP_k^1$ or to a blow-up of $\PP_k^2$ at one point.
\item[(iii)] If $7\geq d\geq 1$, then $X$ is isomorphic to a blow-up of $9-d$ closed points in general position.
\end{enumerate}
Conversely, any surface described under $(i),(ii),(iii)$ for $d\geq 3$ is a del Pezzo surface of the corresponding degree.
\end{thm}

We're only going to deal with del Pezzo surfaces with very ample anticanonical sheaf. We are going to call them \emph{strong del Pezzo surfaces}. The following theorem characterize them:

\begin{thm}(cfr. \cite[Chapter IV, Theorem 24.5]{Man}).
If the surface $X$ is obtained from $\PP_k^2$ by blowing up $r\leq 6$ closed points in general position, then $-K_X$ is very ample and its global sections yield a closed embedding of $X$ in a projective space of dimension

$$ \di \Hl^0(X,\odi{X}(-K_X))-1=K_X^2=9-r.
$$

The set of exceptional curves is identified under this embedding with the set of lines in the projective space which lie on $X$.
The image of $X$ has degree $9-r$.
\end{thm}

\begin{cor}
Let $X$ be a strong del Pezzo surface. Then $X$ is isomorphic either to $\PP_k^1\times\PP_k^1$ or to the blow-up of $r$ points in general position on $\PP_k^2$ for $r=0,\ldots,6$.
\end{cor}

In the next theorem we're going to recall the classical fact that del Pezzo surfaces fall in the class of ACM varieties (cfr. \cite[Expos\'{e} V, Th\'{e}or\`{e}me 1]{Dem}), but before let us recall an important definition that is going to be used through out the paper:
\begin{dfn}(cfr. \cite[Chapter I, Definition 1.1.4]{Mig}).
A coherent sheaf $\shE$ on $\PP_k^n $ is said to be \emph{$m$-regular} if $\Hl^i(\PP_k^n,\shE(m-i))=0$ for all $i>0$.
\end{dfn}

\begin{thm}(cfr. \cite[Expos\'{e} V, Th\'{e}or\`{e}me 1]{Dem})
Let $X$ be a strong del Pezzo surface of degree $d$ and let's consider its embedding in $\PP^d_k$ through the very ample divisor $-K_X$. Then $X\subseteq\PP^d_k$ is an ACM variety.
\end{thm}
\begin{proof}
	We're going to prove that $\Hl^1_{*}(\odi{X})=0$ and $\Hl^1_{*}(\shI_X)=0$. Then the characterization from Lemma \ref{ACMcharacterization} and the short exact sequence definining the ideal of $X$ will allow us to conclude. Let's define $H:=-K_X$. Since $H^2 = d$ and $H$ is very ample, by the adjunction formula and by \cite[Chapter II, Theorem 8.18]{Harti} we obtain that $H$ is a smooth elliptic curve. In particular, since $K_H\sim 0$, from duality we obtain
	$$
	  \h^1(\odi{H}(m)) = \h^0(\odi{H}(-m)) = 0 \text{ for } m>0.
	$$
Since $X$ is rational, we can apply Castelnuovo's criterion to conclude that $\Hl^1(\odi{X})=0$. Next, from the exact sequence
    $$
	  0 \arr \odi{X}(-1) \arr \odi{X} \arr \odi{H} \arr 0
	$$
 twisting by $m\geq 1$ and taking cohomology
	$$
	  \Hl^1(\odi{X}(m-1)) \arr \Hl^1(\odi{X}(m)) \arr \Hl^1(\odi{H}(m)) = 0,
	$$
    we obtain that $\Hl^1(\odi{X}(m))=0$ for any $m\geq 0$. Since $\Hl^1(\odi{X}(m)) \cong \Hl^1(\odi{X}(-m-1))$, the vanishing holds for all $m$.

	It remains to prove that $\Hl^1_*(\shI_X) = 0$; let's consider the exact sequence
	$$
	  0 \arr \shI_X \arr \odi{\PP^d_k} \arr \odi{X} \arr 0.
	$$
	Since $\Hl^2(\odi{X}(2-2)) \cong \Hl^0(\odi{X}(-1))=0$, $\odi{X}$ is $2$-regular. Being $\odi{\PP^d_k}$ $3$-regular, we have that $\shI_X$ is $3$-regular and so $\Hl^1(\shI_X(m)) = 0$ for $m \geq 2$. Clearly this also holds for $m \leq 0$. Finally $\Hl^1(\shI_X(1)) = 0$ since $X$ is embedded through the complete linear system $|-K_X|$.
\end{proof}

Since we're going to accomplish some demonstrations by induction on the degree of the del Pezzo surfaces, the following result will reveal very useful:

\begin{rem}(cfr. \cite[Chapter IV, Corollary 24.5.2]{Man}).\label{blow-down}
	If $X$ is a strong del Pezzo surface and $\pi: X\rightarrow Y$ is a blow-down of a line, then $Y$ is a del Pezzo surface with $H_Y^2 = H_X^2 + 1$.
\end{rem}

To finish this section, let's state an important feature of the ACM bundles on del Pezzo surfaces:

\begin{rem}\label{ACM properties for the dual times canonical}
Let $X$ be a strong del Pezzo surface and $\shE$ be a bundle on it. Then $\E$ is ACM if and only if $\duale{\E}$ is ACM.
\end{rem}

\section{Geometry on strong del Pezzo surfaces}

We're going to work uniquely with strong del Pezzo surfaces, i.e., those del Pezzo surfaces with very ample anticanonical divisor $-K_X$. The goal of this section is to develop Bertini-like theorems for divisors on this kind of varieties. In order to achieve it firstly we will need a good understanding of the exceptional divisors of such varieties. Most of the results presented on this section should be well-known to the specialists but we gather them here for the reader's convenience.
\subsection{Intersection theory}
Let's start stressing a fact that had already been mentioned in the previous section:

\begin{prop}
Let $X$ be a del Pezzo surface and let $C$ be any irreducible smooth curve on $X$. The following conditions are equivalent:
\begin{enumerate}
\item[(i)] $C$ is an exceptional curve (i.e., $C^2=-1$ and $C\cong\PP_k^1$).
\item[(ii)] $C$ is a curve of arithmetic genus $0$ such that $C.K_X=-1$.
\item[(iii)] Let $i:X\hookrightarrow\PP_k^d$ be be the embedding given by the very ample anticanonical divisor $-K_X$. Then $i(C)\subseteq\PP_k^d$ is (an usual) line.
\end{enumerate}
\end{prop}
\begin{proof}
It's a direct computation from the adjunction formula.
\end{proof}

Therefore, since we're only going to deal with del Pezzo surfaces, we're going to use the following convention: we're going to call a curve $C$ in
$X$ an \emph{exceptional divisor} only when we will have fixed a blow-down morphism $\pi:X\rightarrow\PP_k^2$ such that $C$ corresponds to
the inverse image of one of the base points of $\pi$. On the other hand, any curve $C$ verifying the equivalents conditions of the previous proposition will be called a \emph{$(-1)$-line}.

In the following theorem we summarize the well-known results about the Picard group and the intersection product of blow-ups:
\begin{thm}(cfr. \cite[Chapter V, Prop. 4.8]{Harti}). \label{picblowups}
Let $\{p_1,\ldots ,p_r\}$  be a set of points in $\PP_k^2$ and let $\pi:X\rightarrow\PP_k^2$ be the blow-up of $\PP_k^2$ at
these points; let $l\in Pic(X)$ be the pull-back of a line in $\PP^2$, let $E_i$ be the exceptional curves (i.e., $\pi(E_i)=p_i$) and let $e_i\in Pic(X)$ be their linear equivalence classes. Then:

\begin{enumerate}
\item[(i)] $Pic(X)\cong\mathbb{Z}^{r+1}$, generated by $l,e_1,\ldots ,e_r$.
\item[(ii)] The intersection pairing on $X$ is given by $l^2=1$,$e_i^2=-1$, $l.e_i=0$ and $e_i.e_j=0$ for $i\neq j$.
\item[(iii)] The canonical class is $K_X=-3l+\sum_{i=1}^r e_i$.

\end{enumerate}
Moreover, if $0\leq r\leq 6$ and the points are in general position, the following holds:
\begin{enumerate}
\item[(iv)] The anticanonical divisor $H_X=-K_X$ is very ample.
\item[(v)] If $D$ is any effective divisor on $X$, $D\sim al-\sum b_ie_i$ then the degree of $D$ as a curve embedded in $\PP_k^{9-r}$ by $H_X$ is
$deg(D):=3a-\sum b_i$ and its self-intersection is $D^2=a^2-\sum b_i^2$.
\item[(vi)] The arithmetic genus of $D$ is

$$p_a(D)=\frac{1}{2}(D^2-deg(D))+1=\frac{1}{2}(a-1)(a-2)- \frac{1}{2}\sum b_i(b_i -1).
$$
\end{enumerate}

\end{thm}

\begin{rem}(cfr. \cite[Chapter V, Remark 4.8.1]{Harti}). \label{bi}
Using the same notation as in the previous Theorem, if $C$ is any irreducible curve on $X$, other than the exceptional ones $E_i$, then $C_0:=\pi(C)$
is an irreducible plane curve and $C$ in turn is the strict transform of $C_0$. Let $C_0$ have degree $a$ and multiplicity $b_i$ at each $p_i$. Then $\pi^*C_0=C+\sum b_iE_i$. Since $C_0$ is linearly equivalent to $a$ times the class of a line on $\PP^2$, we get $C\sim al-\sum b_ie_i$ with $a> 0$ and $b_i\geq 0$.
\end{rem}

\begin{rem}[Riemann-Roch for divisors on a del Pezzo surface]
	Let $X$ be a del Pezzo surface. Since $X$ is an ACM and connected surface we have $\chi(\odi{X}) = 1$. In particular Riemann-Roch formula for a divisor $D$ has the form
	$$
	  \chi(D) = \frac{D(D+H)}{2} + 1.
	$$
\end{rem}

\begin{lemma}\label{Effectivity condition for ACM divisors}
	Let $X$ be a del Pezzo surface and $D$ be a divisor. If
	$$
	  D^2=D.H-2 \text{ and } D.H > 0
	$$
then $D$ is effective.
\end{lemma}
\begin{proof}
	Suppose by contradiction that $\h^0(D) = 0$. We also have
	$$
	  (-D-H).H < 0 \then \h^2(D) = \h^0(-D-H) = 0.
	$$
	So we obtain the contradiction
	$$
	   D.H=D(D+H)/2 + 1 =\chi(D)=- \h^1(D)  \leq 0.
	$$
\end{proof}

The case of a del Pezzo surface which is the blow-up of one single point in $\PP^2_k$ deserves a special study. The notation of the following remark will be used through out the rest of the paper.
\begin{rem}(cfr. \cite[Chapter V, Proposition 2.3, Corollary 2.11]{Harti}).\label{x1}
	The blow-up of one single point in $\PP_k^2$ can also be interpreted as the rational ruled surface $\pi:X_1 = \PP(\odi{\PP^1_k} \oplus \odi{\PP^1_k}(-1)) \rightarrow \PP^1_k$; write $C_0$ and $f$ for a section and a fibre of $\pi$, respectively. Then $C_0,f$ form a basis of $\Pic X$ and the intersection theory on $X_1$ is given by the relation $C_0^2=-1$, $C_0.f = 1$ and $f^2=0$, while the canonical divisor is $K:=-2C_0 - 3f$. In particular $K^2 = 8$. So $C_0$ is a rational curve with $C_0^2=-1$ and $C_0.H = 1$. It's going to be seen in Proposition \ref{lines} that it is the unique $(-1)$-line on $X_1$. By Remark \ref{blow-down} the contraction of $C_0$ gives us a blow-down morphism $X_1\longrightarrow\PP^2_k$ for which $C_0=e_1$ is the exceptional divisor. Moreover
	$$
	  H_{X_1} = 2C_0 + 3f = 3l - e_1 \then f=l-e_1.
	$$
	Write $D=aC_0 + bf = bl-(b-a)e_1$ for a divisor on $X_1$. Then $D$ is effective if and only if $a=D.(l-e_1) = D.f\geq 0$ and $b=D.l=D.(f+C_0) \geq 0$. Clearly the inequalities imply that $D$ is effective. Conversely if $D$ is effective and $D.f = a < 0$, then a curve in $|D|$ contains all the curves in $|f|$, which is impossible since the union of these curves contains all the closed points of $X$. Finally if $D.l = b < 0$ then a curve in $|D|$ contains all the curves in $|l|$, which is impossible since the union of these curves contains all the closed points of $X - e_1$.
\end{rem}

In the following remark we deal with the quadric case and we introduce the notation that will be use through out the rest of paper.
\begin{rem}\label{effective divisor on quadric}
	Let $X=\PP^1_k \times \PP^1_k$ and $h,m$ be the usual basis of $\Pic X$. Then a divisor $D=ah+bm$ is effective if and only if it's generated by global sections if and only if $a,b \geq 0$. Clearly the inequalities imply that $D$ is effective and generated by global sections. Conversely, if for an effective divisor we had $D.m=a < 0$ that would mean that a curve in $|D|$ contains any curve of $|m|$, which is impossible since the union of these curves contains any closed point of $X$.
\end{rem}

\subsection{$(-1)$-lines on del Pezzo surfaces}

In order to have a good understanding of the properties of the del Pezzo surfaces it's important to keep track of the $(-1)$-lines present on them. This subsection collects some well-known results on their behavior. To start with, the following proposition determines the number of $(-1)$-lines:

\begin{prop}(cfr. \cite[Chapter V, Theorem 4.9]{Harti} and \cite[Expos\'e II, Table 3]{Dem}).\label{lines}
	$\phantom{}$
	\begin{enumerate}
		\item $\PP^1_k \times \PP^1_k$ and $\PP^2$ have no $(-1)$-lines.
		\item Let $X$ be a strong del Pezzo surface which is a blow-up of $r$ points of $\PP^2_k$ in general position, with $1 \leq r \leq 6$. The $(-1)$-lines of $X$ are
		\begin{itemize}
			\item the $r$ exceptional divisors $e_1,\dots,e_r$,
			\item for $r \geq 2$, $F_{i,j} = l - e_i - e_j$ with $1 \leq i < j \leq r$, which are $r(r-1)/2$,
			\item for $r=5$, $G=2l - e_1 - e_2 - e_3 - e_4 - e_5$,
			\item for $r=6$, $G_j = 2l - \sum_{i \neq j} e_i$, which are $6$.
		\end{itemize}
		So $X$ has exactly $r + \binom{r}{2} + \binom{r}{5}$ $(-1)$-lines.
	\end{enumerate}
\end{prop}
	
	


\begin{prop}(cfr. \cite[Chapter V, Proposition 4.10]{Harti}).
Let $X$ be a del Pezzo surface of degree $d$ and set $r=9-d$. If $L_1, \dots, L_r$ are mutually disjoint $(-1)$-lines of $X$ then there exists a blow-up $\pi : X \arr \PP^2_k$ of $r$ points in general position such that $L_1,\dots,L_r$ are the exceptional divisors.
\end{prop}
\begin{proof}
Let $\pi : X \arr Y$ be the blow-down of $L_1, \dots, L_r$. According to Remark \ref{blow-down}, $Y$ is a del Pezzo surface of degree $d+r=9$ and so $Y \cong \PP^2_k$. Following Theorem \ref{picblowups}, if we put $e_i:=L_i$ we know that $K_X=-3l+\sum e_i$. We want now to prove that the points $\{p_1, \dots, p_r\}$ of $\PP^2_k$ image under $\pi$ of $L_1,\dots,L_r$ are in general position, i.e. that no three of them are collinear and no six of them lie on a conic. This can be done as in \cite[Expos\'e II, Th\'eor\`eme 1]{Dem}:
if $p_1,\ldots p_s$ lay on a line, for $s\geq 3$ then its strict transform $D:=l-e_1-\ldots e_s$ would be an effective divisor and $-K_X.D\leq 0$ would contradict the fact that we're supposing that $X$ is a del Pezzo surface and in particular $-K_X$ is very ample.
Analogously, if $p_1,\ldots p_6$ lay on a conic, $D:=2l-e_1-\ldots e_6$ would be an effective divisor such that $-K_X.D\leq 0$, a contradiction.
\end{proof}

    \begin{cor}\label{make two lines exceptional divisors}
    	Let $X$ be a del Pezzo surface of degree $d$ and $L, L'$ be skew $(-1)$-lines of $X$. If $r:=9-d\geq 4$, then $L, L'$ are exceptional divisors for some blow-up $X \arrdi{\pi} \PP^2_k$ of $r$ points in general position.
    \end{cor}
    \begin{proof}
	Since we're supposing that there exist two skew $(-1)$-lines, we know that $X$ is the blow-up of $r$ points in general position on $\PP_k^2$. Therefore it's enough to show that $L, L'$ are contained in a set of $r$ mutually skew $(-1)$-lines of $X$. If $L$ and $L'$ are already part of the exceptional divisors of the blow-up morphism that we're considering, we're done. If it's not the case, with regard to the notation of Proposition \ref{lines}, up to permutation of the exceptional divisors, it's straightforward to check that they form part of one of the following sets of skew $(-1)$-lines:
	\begin{align*}
         &F_{1,2}, F_{1,3}, F_{2,3}, e_4 && \text{if $r = 4$,}\\
		 &F_{1,2}, F_{1,3}, F_{1,4}, F_{1,5}, G ; \ \ \ \ F_{1,2}, F_{1,3}, F_{2,3}, e_4, e_5 && \text{if $r = 5$,}\\
		 &F_{1,2}, F_{1,3}, F_{1,4}, F_{1,5}, G_6, e_6 ; \ \ \ \ F_{1,2}, F_{1,3}, F_{2,3}, e_4, e_5 , e_6& &\text{if $r = 6$.}
	\end{align*}
    \end{proof}

\subsection{Very ample and smooth divisors}
In this subsection we give criterions in terms of the intersection with $(-1)$-lines for a linear system to be very ample or at least to contain smooth representatives.
    \begin{lemma}(cfr. \cite[Chapter V, Lemma 4.12]{Harti}).\label{alternative base for Pic X}
	Let $X$ be a del Pezzo surface which is a blow-up of $r$ points of $\PP^2_k$ in general position, for $2 \leq r \leq 6$, and let's consider the divisors $D_0,\ldots, D_r$ defined as follows:
	\begin{align*}
		D_0 & = l, \\
		D_1 & = l-e_1, \\
		D_2 & = 2l - e_1 - e_2, \\
		D_3 & = 2l - e_1 - e_2 - e_3, \\
		D_4 & = 2l - e_1 - e_2 - e_3 - e_4, \\
		D_5 & = 3l - e_1 - e_2 - e_3 - e_4 - e_5, \\
		D_6 & = 3l - e_1 - e_2 - e_3 - e_4 - e_5 - e_6.
	\end{align*}
	Then $D_0,\dots,D_r$ are effective divisors without base points in $X$ and form a basis of $\Pic X$. If $D=al - \sum_i b_ie_i$ is any divisor in $X$ then
	$$
		  D = \alpha D_0 + \sum_{i=1}^{r-1}(b_i - b_{i+1})D_i + b_rD_r
	$$
	where
	$$
	 \alpha = \left\lbrace \begin{array}{lc}
	              	a -b_1 - b_2 = D.F_{1,2} & \text{if } 2 \leq r \leq 4 \\
			a -b_1 - b_2 -b_5= D.F_{1,2}-D.e_5 & \text{if } 5 \leq r \leq 6.
	              \end{array}
		      \right.
	$$
    \end{lemma}

\begin{proof}
	$D_0,\dots,D_r$ form a base of $\Pic X$ because they are the image of the base $l,e_1,\dots,e_r$ with respect to an invertible matrix with determinant $\pm 1$. They are without base points thanks to \cite[Chapter V, Proposition 4.1 and Proposition 4.3]{Harti}. A direct computation provides the last equality.
\end{proof}

\begin{cor}\label{decomposition of a divisor with respect to lines}
With the same notation and hypothesis of Lemma \ref{alternative base for Pic X}, for any divisor $D$ of $X$ there's a choice of exceptional divisors in $X$ such that
	$$
		  D = \alpha_0 D_0 + \sum_{i=1}^{r-1}\alpha_i D_i + (D.e_r)D_r
	$$
	where $\alpha_1,\dots,\alpha_{r-1} \geq 0$ and
	$$
	 \alpha_0 \text{ is } \left\lbrace \begin{array}{lc}
	              	= D.F_{1,2} & \text{if } 2 \leq r \leq 4 \\
			\geq 0 & \text{if } 5 \leq r \leq 6.
	              \end{array}
		      \right.
	$$
\end{cor}

    \begin{proof}
      If $2 \leq r \leq 4$ it's enough to relabel the given exceptional divisors so that
      $$
      D.e_1 \geq \cdots \geq D.e_r.
      $$
      If $5 \leq r \leq 6$, we can proceed in this way. Choose a line $L$ such that
      $$
      D.L = \min \{ D.L' \ | \ L' \text{ line of } X \}.
      $$
      Note that no $(-1)$-line of $X$ meets all the other $(-1)$-lines of $X$ and so we can choose a second $(-1)$-line $L'$ such that
      $$
      D.L' = \min \{ D.L'' \ | \ L'' \text{ line of } X \text{ such that } L'.L'' = 0 \}.
      $$
      From Corollary \ref{make two lines exceptional divisors} we can assume that $L, L'$ are exceptional divisors, namely $e_r=L$ and $e_{r-1} = L'$. As above we can relabel $e_1, \dots , e_{r-2}$ so that $D.e_1 \geq \cdots \geq D.e_r$. Finally, since $F_{1,2}.e_r = 0$, we have
      $$
	\alpha_0 = D.F_{1,2} - D.e_5 \geq 0.
      $$
    \end{proof}

The next lemma gives a nice criterion in order to know when a divisor is very ample:

\begin{lemma}(cfr. \cite[Chapter V, Theorem 4.11]{Harti}).\label{veryample}
	Let $X$ be a del Pezzo surface which is a blow-up of $r$ points in general position of $\PP^2_k$ and $D$ be a divisor on $X$. If $2 \leq r \leq 6$, $D$ is very ample if and only if $D.L > 0$ for any $(-1)$-line $L$ on $X$. If $r=1$ then $D$ is very ample if and only if $D.e_1, D.(l-e_1) > 0$.
\end{lemma}
\begin{proof}
	From the Nakai-Moishezon criterion the inequalities hold if $D$ is very ample. So we focus on the converse. If $r=1$ and $D=aC_0 + bf$ the two conditions say that $D.C_0 = D.e_1 = b -a > 0$ and $D.(l-e_1)=D.f = a > 0$ and the result follows from \cite[Chapter V, Corollary 2.18.]{Harti}.  If $2 \leq r \leq 6$, according to Corollary \ref{decomposition of a divisor with respect to lines}, we can write
	$$
	  D = \alpha_0 D_0 + \cdots \alpha_r D_r  \text{ with } \alpha_1, \dots, \alpha_{r-1} \geq 0 \text{ and } \alpha_r = D.e_r > 0;
	$$
	if $2 \leq r \leq 4$, we have that  $\alpha_0 = D.F_{1,2} > 0$ and therefore, since $D_r + D_0 = H_X$, $D$ is $H_X$ plus a sum of divisors generated by global sections. On the other hand, if $5 \leq r \leq 6$ then $\alpha_0 \geq 0$ and $D_r = H_X$. In any case, since a very ample divisor plus a divisor generated by global sections is very ample, we are done.
\end{proof}

\begin{rem}
	If $X=X_1$ is the blow-up of one point of $\PP^2_k$, then for a divisor $D=aC_0 + bf$ the condition $D.e_1 = b-a > 0$ ($e_1=C_0$ is the unique line of $X$) is not enough for ampleness. For example $D=-C_0 + f$ is not effective, while $D=f$ is effective but $D.f=f^2 = 0$ and so it's not ample.
\end{rem}

\begin{thm}\label{smoothness criterion}
	Let $X$ be a del Pezzo surface and $D$ be a non zero effective divisor. Then $D.L \geq 0$ for any $(-1)$-line $L$ of $X$ if and only if the linear system $|D|$ contains an open non-empty subset of smooth curves with no $(-1)$-lines as irreducible components. Such a divisor is always generated by global sections.
\end{thm}
\begin{proof}
	$\Longleftarrow )$ If $D$ is smooth and $D.L < 0$ for some $(-1)$-line $L$ then $L$ is in the base locus of $|D|$ and therefore $L$ is an irreducible component of any element of this linear system.
	
	$\then)$ Clearly if we prove that $|D|$ contains an open subset of smooth curves $C$, the same argument used above shows that $C$ doesn't contain a line.
	
	The cases $X=\PP^1_k \times \PP^1_k$ and $X=\PP_k^2$ don't contain any $(-1)$-line and therefore we would need to prove that the linear class of any non-zero effective divisor contains an open non-empty set of smooth curves and it's generated by global sections, which is a very well-known fact.

	$X = X_1$, i.e. $X$ is the blow-up of a point in $\PP^2_k$. Write $D=aC_0 + bf$, with $a,b \geq 0$. We have $D.C_0 = b-a \geq 0$. If $a=0$, then $D=bf$ is a disjoint union of $b$ distinct fibers of the usual projections $X_1 \arr \PP^1_k$ and $f$ is generated by global sections. If $a = D.f > 0$ and $D.C_0 > 0$ we know that $D$ is very ample and therefore we can apply Bertini's theorem to get the conclusion. It remains the case $b=a > 0$, i.e. $D=a(C_0 + f) = al$. Since $l$ is generated by global sections so is $D$. Finally $|al|$ contains the inverse image of the open non-empty set of curves of (usual) degree $a$ in $\PP^2_k$ which don't contain the point blown up.
	
	In order to treat the remaining cases, we proceed by descent induction on the degree of $X$ and therefore we can suppose that $X$ is a blow-up of $r$ points of $\PP^2_k$, with $2 \leq r \leq 6$. If $D.L > 0$ for any $(-1)$-line of $X$ we know that $D$ is very ample (see Lemma \ref{veryample}) and therefore by Bertini's theorem we get the conclusion. So suppose that $D.L = 0$ for some $(-1)$-line $L$ of $X$. From Corollary \ref{decomposition of a divisor with respect to lines} we see that $D$ is generated by global sections. So we can take $C \in |D|$ such that $C \cap L = \emptyset$, the blow-down $X \arrdi{\pi} Y$ with respect to $L$ and consider the divisor $D':=\pi C$ on $Y$. Thanks to Remark \ref{blow-down}, we know that $Y$ is a del Pezzo surface of degree $H_Y^2=H_X^2+1$.

    If it was the case that there exists a $(-1)$-line $L'$ on $Y$ such that $D'.L' < 0$ then $L' \subseteq D'$ and in particular $L'$ doesn't contain the point $\pi(L)$. Therefore $\pi^*L'$ is a $(-1)$-line of $X$ and $\pi^*L'.D = \pi^*L'.\pi^*D' = L'.D' < 0$, which is a contradiction. Therefore $D'.L' \geq 0$, for any $(-1)$-line $L'$ in $Y$ and we can apply the hypothesis of induction to $Y$ to get an open non-empty subset $U$ of the linear system $|D'|$ composed of smooth curves with no $(-1)$-lines as irreducible components. Since $|D'|$ is generated by global sections, the point $\pi(L)$ is not a fixed point of this linear system and therefore we can suppose, restricting the open set if necessary, that no curve of $U$ passes  through $\pi(L)$. Then $\pi^{*}$ gives us the open non-empty set of smooth curves without $(-1)$-lines as components on $D$.

\end{proof}

\section{Classification of ACM line bundles}

In this section $X$ will be a strong del Pezzo surface of degree $d=3,\ldots ,9$ embedded in $\PP_k^d$ by the very ample divisor $-K_X$. In particular, when we will speak of ACM bundles on $X$ it will always be with respect to this divisor. We follow notation from Theorem \ref{picblowups}.

\subsection{Geometrical characterization of ACM line bundles}
The goal of this subsection will be characterize numerically ACM line bundles on del Pezzo surfaces. Moreover we're going to show that they correspond to rational normal curves on the surface.

\begin{rem}
	Let $D$ be a non zero effective divisor on a del Pezzo surface $X$ and consider the exact sequence
	    $$
	      0 \arr \odi{X}(-D) \arr \odi{X} \arr \odi{D} \arr 0.
	    $$
	    Since $\h^0(\odi{X}(-D)) = \h^1(\odi{X}) = \h^2(\odi{X}) = 0$, taking cohomology we obtain the two exact sequences
	    \begin{align*}
	    	0 \arr \Hl^0(\odi{X}) \arr \Hl^0(\odi{D}) \arr \Hl^1(\odi{X}(-D)) \arr 0, \\
		\end{align*}
        and
        \begin{align*}
             0 \arr \Hl^1(\odi{D}) \arr \Hl^2(\odi{X}(-D)) \arr 0,
	    \end{align*}
	    and therefore the equalities
	    \begin{align}\label{first equalities for an effective divisor}
	    	\h^0(\odi{D}) = 1 + \h^1(\odi{X}(-D)), \ \ \ \ \h^1(\odi{D})=\h^2(\odi{X}(-D)).
	    \end{align}
\end{rem}

\begin{prop}\label{degeneration}
Let $X\subseteq\PP_k^d$ be a del Pezzo surface of degree $d$ and let $\odi{X}(D)$ be an initialized line bundle on $X$ with $D$ a rational smooth curve of degree $c$. Then $D$ is a non-degenerate curve on some $\PP_k^m$ for $m=c-\h^1(\odi{X}(D-2H))$.
\end{prop}

\begin{proof}
The statement is reduced to the computation of the dimension of $\Hl^0(\shI_{D|\PP_k^d}(1))$ and it's performed as follows: let's consider the exact sequence
$$
0\arr \shI_{X| \PP_k^d} \arr\shI_{D| \PP_k^d} \arr\shI_{D|X} \arr 0.
$$
Since $X$ was non degenerate and ACM, applying the functor of global sections to the previous sequence twisted by $\odi{\PP_k^d}(1)$ we get that
$$
\Hl^0(\shI_{D|\PP_k^d}(1))\cong \Hl^0(\shI_{D|X}(1))=\Hl^0(\odi{X}(-D+H)).
$$
On the other hand, by Riemann-Roch,
$$
\chi(-D+H)=\frac{1}{2}(-D+H)(-D+2H)+1=d-c,
$$
using the fact that $D$ is smooth and rational. Since $\odi{X}(D)$ was initialized, $\h^2(-D+H)=\h^0(D-2H)=0$ and we can conclude that
$$
\h^0(-D+H)=d-c+\h^1(-D+H).
$$

\end{proof}

\begin{dfn}
	A rational normal curve of degree $d$ is a non-degenerate rational smooth curve of degree $d$ in some $\PP_k^d$.
\end{dfn}

\begin{rem}\label{fact on rational normal curves}
	If $D$ is a rational normal curve of degree $d$ then
	$$
	  \Hl^0(\odi{\PP^d_k}(1)) \arr \Hl^0(\odi{D}(1)) \cong \Hl^0(\odi{\PP^1_k}(d))
	$$
	is injective and therefore an isomorphism. This means that $D$ is embedded through the complete linear system $|\odi{\PP^1_k}(d)|$ and so, up to automorphism of $\PP^d_k$, is unique. A classical result, which can be found in \cite[Corollary 6.2]{EisSyz}, is that $D$ is ACM in $\PP^d_k$.
\end{rem}

    \begin{thm}\label{maintheorem}
	    Let $X\subseteq\PP^d_k$ be a del Pezzo surface of degree $d$ embedded through the very ample divisor $-K_X$ and $\shL$ be a line bundle on $X$.
	    They are equivalent:
	    \begin{enumerate}
	    \item $\shL$ is initialized and ACM.
		\item $\shL$ is initialized and $\h^1(\shL(-1))=\h^1(\shL(-2))=0$.
		\item $\shL \cong \odi{X}$ or $\shL \cong \odi{X}(D)$, where $D$ is a divisor such that $D^2=D.H-2$ and $0<D.H \leq H^2$.
		\item $\shL \cong \odi{X}$ or $\shL \cong \odi{X}(D)$, where $D$ is a rational normal curve on $X$ with $\deg D \leq d$.
	    \end{enumerate}
    \end{thm}
    \begin{proof}
	
Before starting the prove, we want to give some general remarks. Let $D$ be a non zero effective divisor and $\shL \cong \odi{X}(D)$.
	    First of all, from (\ref{first equalities for an effective divisor}), we have
	    \begin{align}\label{h0 of o of D}
	    	\h^0(\odi{D}) = 1 \iff \h^1(\shL^{-1})=\h^1(\shL(-1)) = 0.
	    \end{align}
	    Moreover if $\shL$ is initialized then
	    \begin{align}\label{h1 of o of D and L minus one}
	    	\h^1(\shL(-1))=\frac{D.H-D^2}{2}-1 \text{ and } \h^1(\odi{D})=0.
	    \end{align}
	    Indeed, again by (\ref{first equalities for an effective divisor}), $\h^1(\odi{D})=\h^2(\shL^{-1}) = \h^0(\shL(-1))=0$ and, since $\shL^{-1}$ is a proper sheaf of ideals of $\odi{X}$ and so $\h^2(\shL(-1))= \h^0(\shL^{-1})=0$ we get
	    \begin{align*}
	    	-\h^1(\shL(-1)) = \chi(\shL^{-1}) = \chi(-D) = \frac{D^2 - D.H}{2} + 1.
	    \end{align*}

Now we can start proving the equivalences:

$4) \then 1)$. Since $\odi{X}$ is ACM and initialized, we consider the case $\shL \cong \odi{X}(D)$, where $D$ is a rational normal curve of degree $D.H=c \leq d$. Clearly $\shL$ has global sections. Moreover from the adjunction formula we have $D^2=c-2$ and so from Riemann-Roch
	    $$
	      \chi(\shL(-1))=\chi(D-H) = (D-H)D/2+1 = 0.
	    $$
$\shL^{-1}:=\odi{X}(-D)$, being the ideal sheaf of $D$, has no global sections: $\h^2(\shL(-1))=\h^0(\shL^{-1})=0$. Therefore, if we prove that $\shL$ is ACM, we also get that $\shL$ is initialized.
	    From Lemma \ref{ACMvar} and Remark \ref{ACM properties for the dual times canonical} this is equivalent to prove that the rational normal curve $D$ is ACM, which is a classical fact (see Remark \ref{fact on rational normal curves}).

	    $1) \then 2)$ It's clear.

	    $2) \then 3)$ Let $D$ be an effective divisor such that $\shL \cong \odi{X}(D)$ and suppose $D \nsim 0$.
	    By (\ref{h1 of o of D and L minus one}) we get the equality $D^2=D.H-2$. Since $\h^0(\shL(-2))=\h^1(\shL(-2))=0$ we also have
	    $$
	      0 \leq  \h^2(\shL(-2)) = \chi(D-2H) = H^2 - D.H \then D.H \leq H^2.
	    $$
	    Finally, since $D$ is a non zero effective divisor, we have $D.H > 0$.

	    $3) \then 4)$ Let's set $\shL=\odi{X}(D)$, with $D \nsim 0$. By Lemma \ref{Effectivity condition for ACM divisors} we know that $D$ is effective. So let's also set $c:=D.H = \deg D$.

       We begin showing that $\shL$ is initialized. Otherwise suppose that $D-H$ is effective. Since $D^2=D.H-2$ we
       know that $D-H$ is non zero and so
	    $$
	    0 < (D-H).H = D.H - H^2 \leq 0
	    $$
would give us a contradiction.
	    Therefore, from (\ref{h1 of o of D and L minus one}) and (\ref{h0 of o of D}), we obtain
	    $$
	    \h^1(\shL(-1))=\h^1(\odi{D})=0 \text{ and } \h^0(\odi{D})=1.
	    $$
	    So if $|D|$ contains a smooth curve $C$, $C$ is connected, has genus $0$, i.e. it's rational, and $\deg C = D.H = c \leq H^2$.
	
	    Let's prove that $|D|$ contains a smooth curve: from Theorem \ref{smoothness criterion} we know that if $D.L \geq 0$ for any line $L$ of $X$, then $D$ contains a smooth curve. So we want to prove that if $L$ is a line of $X$ such that $D.L < 0$, then $D=L$.
	    Write $M=D-L$. $M$ is an effective divisor and suppose, by contradiction, that $M \nsim 0$. Note that, since $M^2=(D-L)^2 = c-3-2D.L$ and $M.H=c-1$, we obtain by Riemann-Roch
	    $$
	    \chi(-M) = 1+ (-M)(-M+H)/2 = 1 + (c-3 - 2D.L -c +1)/2 = -D.L > 0.
	    $$
	
        But, on the other hand, from the exact sequence
	    $$
	      0 \arr \shI_{M|D} \arr \odi{D} \arr \odi{M} \arr 0
	    $$
	    we get
	    $$
	      \Hl^1(\odi{D}) = 0 \arr \Hl^1(\odi{M}) \arr \Hl^2(\shI_{M|D}) = 0
	    $$
        where the last cohomology groups vanishes because $\shI_{M|D}$ is a sheaf on a one dimensional variety. Therefore, from (\ref{first equalities for an effective divisor}), we obtain $\h^2(\odi{X}(-M)) = \h^1(\odi{M}) = 0$.
	    Moreover, since $M \nsim 0$, we have $\h^0(\odi{X}(-M)) = 0$ and therefore
        $$
        \chi(-M)=-h^1(\odi{X}(-M))\leq 0
        $$
        which is obviously a contradiction.

Therefore we know that we can take $C\in |D|$ a smooth rational curve. In order to see that $C$ is a rational normal curve it's enough to prove that $\h^0(\shI_{C|\PP_k^d}(1))=d-c$ because then $C$ will be a non-degenerate rational curve on $\PP_k^c$ of degree $c$. As in Proposition \ref{degeneration}, this number is $\h^0(\odi{X}(-D+H))$. So let's consider the divisor $E=-D+H$. It has the following invariants:
$$
E.H=d-c\leq d,
$$
and
$$
E^2=D^2+H^2-2D.H=E.H-2.
$$
If $c=d$, since $D\nsim H$, $E$ can not be effective and therefore $\h^0(\odi{X}(E))=0$. Otherwise, if $c<d$, we have seen that under this conditions we can suppose that $E$ is a smooth rational curve. From the exact sequence

$$
0\arr\odi{X}(-E)\arr\odi{X}\arr\odi{E}\arr 0
$$
if we twist it by $\odi{X}(E)$ and take cohomology we get

$$
0\arr \Hl^0(\odi{X})\arr \Hl^0(\odi{X}(E))\arr \Hl^0(\odi{E}(E))\arr \Hl^1(\odi{X})=0.
$$
The degree of $\odi{E}(E)$ is $E^2=d-c-2$ and therefore $\h^0(\odi{E}(E))=\h^0(\odi{\PP_k^1}(d-c-2))=d-c-1$ and so $\h^0(\odi{X}(E))=d-c$.

\end{proof}

\subsection{Explicit list of ACM divisors}

Once we know how to characterize ACM line bundles on del Pezzo surfaces, this subsection will be dedicated to list them: first in the case of the quadric and then in the rest of cases consisting on blow-ups.

\begin{lemma}
   	There exist exactly (up to twist and isomorphism) $8$ initialized ACM line bundles on the del Pezzo $\PP^1_k \times \PP^1_k$ with respect to the very ample divisor $-K_X$. The initialized ones are given by $\odi{\PP^1_k \times \PP^1_k}$ and, in terms of their associated class of divisors,
	$$
	  D=h+bm \text{ or } D=bh+m \text{ with } 0 \leq b \leq 3 \text{ (} \deg D = 2+2b \text{).}
	$$
    \end{lemma}
    \begin{proof}
    	Let $D=ah+bm$ be any divisor. So $D$ is initialized and ACM if and only if $D\sim 0$ or
	\begin{align*}
		0<D.H \leq 8 & \iff 0<2a + 2b \leq 8, \\
		D^2=D.H-2 & \iff 2ab=2a+2b-2 \iff (a-1)(b-1)=0,
	\end{align*}
	that gives exactly the divisors listed in the proposition.
\end{proof}

\begin{thm}\label{explicit computation of initialized ACM line bundles}
	Let $X$ be a del Pezzo surface which is a blow-up of $r$ points on $\PP^2_k$, with $0 \leq r \leq 6$. With respect to $-K_X$, the initialized ACM divisors of $X$ are $0$, the exceptional divisors and, up to permutation of the exceptional divisors, the ones listed below:
\begin{center}
\noindent\begin{tabular}{|c|l|c|}
\hline
     $\deg D$ & D & \\
\hline
     $3-m$ & $l - e_1 - \cdots - e_m$ & $0 \leq m \leq \min\{2,r\}$\\
\hline
     $6-m$ & $2l - e_1 - \cdots - e_m$ & $\max\{r-3,0\} \leq m \leq \min\{5,r\}$\\
\hline
     $8-m$ &  $3l - 2e_1 - e_2 \cdots - e_m$ & $\max\{1,r-1\} \leq m \leq r$\\
\hline		
      $9 - r$ & $4l - 2e_1 - 2e_2 - 2e_3 - e_4 \cdots - e_r$ & $r \geq 3$\\
\hline
      6 & $5l - 2e_1 - 2e_2 - 2e_3 - 2e_4 - 2e_5 - 2e_6$ & $r = 6$\\
\hline
\end{tabular}
\end{center}

	With respect to the degree $d$ we have, up to permutation of the exceptional divisors:
	\begin{center}

\noindent\begin{tabular}{|c|c|l|c|}
\hline
     $d=\deg D$ & $u(D,r)$ & D & r \\
\hline
     0 & 1 & 0 & \\
\hline
     1 & r & $e_1$ & r $\geq$ 1 \\
     & $\binom{r}{2}$ & $l - e_1 - e_2$ & r $\geq$ 2 \\
	 & $\binom{r}{5}$ & $2l - e_1 - e_2 - e_3 - e_4 - e_5$ & $r \geq 5$ \\
\hline
     2 &  r & $l - e_1$ & $r \geq 1$ \\
	 & $\binom{r}{4}$ & $2l - e_1 - e_2 - e_3 - e_4$ & $r \geq 4$ \\
	 & 6 & $3l - 2e_1 - e_2 - e_3 - e_4 - e_5 - e_6$ & $r = 6$ \\
\hline		
	    3 & 1 & l & \\
		& $\binom{r}{3}$ & $2l - e_1 - e_2 - e_3$ & $r \geq 3$ \\
		& $r\binom{r-1}{4}$ & $3l - 2e_1 - e_2 - e_3 - e_4 - e_5$ & $r \geq 5$ \\
		& 20 & $4l - 2e_1 - 2e_2 - 2e_3 - e_4 - e_5 - e_6$ & $r = 6$ \\
		& 1 & $5l - 2e_1   - 2e_2 - 2e_3 - 2e_4 - 2e_5 - 2e_6$ & $r = 6$ \\
\hline
		4 & $\binom{r}{2}$ & $2l - e_1 - e_2$ & $2 \leq r \leq 5$ \\
		& $r\binom{r-1}{3}$ & $3l - 2e_1 - e_2 - e_3 - e_4$ & $4 \leq r \leq 5$ \\
		& 10 & $4l - 2e_1 - 2e_2 - 2e_3 - e_4 - e_5$ & $r = 5$ \\
\hline
		5 & r & $2l - e_1$ & $1 \leq r \leq 4$ \\
		& $r\binom{r-1}{2}$ & $3l - 2e_1 - e_2 - e_3$ & $3 \leq r \leq 4$ \\
		& 4 & $4l - 2e_1 - 2e_2 - 2e_3 - e_4$ & $r = 4$ \\
\hline		
	    6 & 1 & 2l & $1 \leq r \leq 3$ \\
		& r(r-1) & $3l - 2e_1 - e_2$ & $2 \leq r \leq 3$ \\
		& 1 & $4l - 2e_1 - 2e_2 - 2e_3$ & $r = 3$ \\
\hline	
		7 & r & $3l - 2 e_1$ & $1 \leq r \leq 2$ \\
\hline
		8,9 & 0 && \\
\hline

\end{tabular}
\end{center}

\noindent where $u(D,r)$ is the number of divisors obtained permuting the exceptional divisors in the writing of $D$.

\end{thm}

\begin{cor}
Write $X^r$ for a del Pezzo surface blow-up of $r$ points of $\PP^2_k$ and $Q$ for the quadric $\PP^1_k \times \PP^1_k$. The table below lists the number of initialized ACM line bundles of a given degree $d \leq \Hl^2_{X^r}$ (resp. $d \leq H_Q^2$) contained in $X^r$ (resp. in $Q$).

\begin{center}

\noindent\begin{tabular}{|c|c|c|c|c|c|c|c|c|c|}
\hline
	 d & $X^r$ & $X^0$ & $X^1$ & $X^2$ & $X^3$ & $X^4$ & $X^5$ & $X^6$ & $Q$\\
\hline
	0&1 & 1 & 1 & 1 & 1 & 1  & 1  & 1& 1\\
\hline
	 1&$r + \binom{r}{2} + \binom{r}{5}$  & 0 & 1 & 3 & 6 & 10 & 16 & 27 &0\\
\hline
	 2&$r + \binom{r}{4} + r\binom{r}{6}$  & 0 & 1 & 2 & 3 & 5  & 10 & 27 &2\\
\hline
	 3&$1 + \binom{r}{3} + r\binom{r-1}{4} + \binom{r}{3}\binom{r-3}{3} + \binom{r}{6}$  & 1 & 1 & 1 & 2 & 5  & 16 & 72 &0\\
\hline
	 4&$\binom{r}{2} + r\binom{r-1}{3} + \binom{r}{2}\binom{r-2}{3}$ & 0 & 0 & 1 & 3 & 10 & 40 &    &1\\
\hline
	 5&$r + r\binom{r-1}{2} + r\binom{r-1}{3}$ & 0 & 1 & 2 & 6 & 20 &    &    &0\\
\hline
	 6&$1 + r(r-1) + \binom{r}{3}$ & 1 & 1 & 3 & 8 &    &    &    &2\\
\hline
	 7&$r$ & 0 & 1 & 2 &   &    &    &    &0\\
\hline
	 8&$0$ & 0 & 0 &   &   &    &    &    &2\\
\hline
	 9&$0$ & 0 &   &   &   &    &    &    &\\
\hline
	 Tot& & 3 & 7 & 15 & 29 & 51 & 83 & 127 &8\\
\hline
\end{tabular}
\label{table of numbers of ACM divisors for a given X and d}
\end{center}

\noindent The formula on the column of $X^r$ for the number of initialized ACM line bundles of degree $d$ makes sense only if $d \leq 9-r$.
\end{cor}

\begin{proof}[Proof of \ref{explicit computation of initialized ACM line bundles}.]
	We will look for ACM initialized line bundles applying condition $(3)$ from Theorem \ref{maintheorem}.
	
	Let $D$ be a divisor and suppose that $D$ is not exceptional. Set also $d = D.H$. Label the exceptional divisors $e_1, \dots, e_r$ in such a way that $D.e_1 \geq \cdots \geq D.e_r$, i.e.
	$$
	  D = al - b_1 e_1 - \cdots - b_r e_r \text{ with } b_1 \geq \cdots \geq b_r.
	$$
	Since we already know the $(-1)$-lines of $X$ and since for an initialized ACM divisor $C$ with $\deg C \geq 2$, $C.L \geq 0$ for any $(-1)$-line $L$ of $X$, we can assume $b_r \geq 0$.
	
	Let $m$ be such that $b_1, \dots, b_m > 0$ and $b_{m+1} = \cdots = b_r = 0$. If $\pi : X \arr \PP^2_k$ is the blow-up that defines $X$, let $Y$ be the blow-up of $\pi(e_1), \dots, \pi(e_m)$ and denote by $l',e_1',\dots ,e_m'$ the usual basis of Pic($Y$). $Y$ is a del Pezzo surface and we have a map $X \arrdi{f} Y$ such that $f^*l'=l, f^*e'_1 = e_1, \dots f^*e'_m = e_m$. If we set
	$$
	  D' = al' - b_1 e'_1 - \cdots - b_m e'_m \in \Pic Y
	$$
	then $D$ is initialized and ACM if and only if $D'$ is so, thanks to Theorem \ref{maintheorem}. In conclusion, we can assume $m = r$, i.e., $b_r > 0$. This means that if we find the initialized ACM divisors in this case, the other ones can be obtained with the same writing, only checking the condition $\deg D \leq H^2$.

	In the case $r=0$, i.e. $X \cong \PP^2_k$, the initialized ACM line bundles are $0, l, 2l$.
	
	Assume now $r=1$ and following notation from Remark \ref{x1} write $D=\alpha C_0+\beta f = \beta l - (\beta - \alpha)e_1$. Since $H=2C_0+3f$ we have
	$$
	  D^2=-\alpha^2+2\alpha \beta, \ \ \ \ \ \ d = 2\beta+\alpha
	$$
	and
	$$
	  D^2=d-2 \iff 2\beta(\alpha-1)=(\alpha-1)(\alpha+2).
	$$
	If $\alpha=1$ then $D=\beta l - (\beta - 1)e_1$ and we have
	$$
	  1\leq d=2\beta+1 \leq H^2=8 \iff 0\leq\beta \leq 3.
	$$
	In this way we obtain divisors $e_1, l, 2l - e_1, 3l - 2e_1$.
	
	If $\alpha \neq 1$, then $\alpha=2\beta-2$, $D= \beta l - (2 - \beta)e_1$ and
	$$
	 1\leq d=4\beta-2 \leq H^2=8 \iff 1\leq\beta \leq 2.
	$$
	So we obtain divisors $l - e_1, 2l$.
	
	We can so assume $r \geq 2$ and, since we have already treated the case $d=1$ in Proposition \ref{lines}, $d \geq 2$. So suppose that $D$ is ACM, initialized with $\deg D = d \geq 2$. The equations $D^2 = d-2$ and $D.H=d$ are
	\begin{align}
		\label{eq: D*D = d - 2} \sum_i b_i^2 & = a^2+2 -d, \text{ and}\\
		\label{eq: D.H = d} \sum_i b_i & = 3a -d.
	\end{align}

	The first step is to prove that $b_1+b_2 \leq a \leq 5$. The first inequality is
	$$
	  D.F_{1,2} = a - b_1 - b_2 \geq 0
	$$
	which holds since $D$ is a rational normal curve of degree $d \geq 2$.
	From (\ref{eq: D*D = d - 2}) and (\ref{eq: D.H = d}) we obtain
	$$
	  (3a - d)^2 = (\sum_i b_i)^2 \leq r\sum_i b_i^2=r(a^2+2-d)\leq (9-d)(a^2+2-d)
	$$
	which is equivalent to
	$$
	  p(a) = d a^2 - 6ad + 11d-18 \leq 0.
	$$
	But
	$$
	  p(0)=p(6)=11d-18 \geq 22-18=4 >0 \then 0 < a < 6.
	$$
	
	In particular we have $b_1 < a \leq 5$. The remaining part of the proof is divided in the following cases: $b_1=4$, $b_1 = 3$ and $b_1 \leq 2$.
	
	$\bullet$ $b_1 = 4$. $a \geq b_1 + b_2 \geq 5$ implies that $a=5$ and $b_2 = \cdots = b_r = 1$. So also this case is impossible since otherwise from (\ref{eq: D.H = d}) we obtain
	$$
	  r-1 = 11-d \then d = 12 - r \leq 9 - r.
	$$
	
	$\bullet$ $b_1 = 3$. Again $a \geq b_1 + b_2$ tells us that $b_2 \leq 2$ and $a=4,5$. We split it in two parts: $b_2 = 1$ and $b_2 = 2$.
	
	$b_2 = 1$: In this case from (\ref{eq: D*D = d - 2}) and (\ref{eq: D.H = d}) we obtain
	$$
	  r-1 = a^2-7-d=3a-3-d \then a^2-7 = 3a - 3 \then a = 4.
	$$
	and the contradiction comes from
	$$
	  r-1 = 9-d \then d=10 -r \leq 9 - r.
	$$
	
	$b_2 = 2$: We have $a=5$ and (\ref{eq: D*D = d - 2}) and (\ref{eq: D.H = d}) become
	$$
	  \sum_{i>2}b_i^2 = 14-d, \ \ \ \sum_{i>2}b_i = 10 -d.
	$$
	Since $14-d \neq 10 -d$ we cannot have $r=2$ or $b_3 = 1$, so $b_3 = 2$. Arguing in this way we also obtain $r \geq 4$, $b_4 = 2$ and the contradiction
	$$
	  \sum_{i>4}b_i^2 = \sum_{i>4}b_i = 6 -d \then r-4 = 6-d \then d=10-r \leq 9-r.
	$$
	
	$\bullet$ $b_1 \leq 2$. Let $s$ be such that $b_1=\cdots=b_s=2$, $b_{s+1}=\cdots=b_r=1$ for $0 \leq s \leq r$.
	The equation (\ref{eq: D*D = d - 2}) and (\ref{eq: D.H = d}) become
	$$
	  4s + (r-s)=a^2+2-d, \ \ \ 2s + (r-s)=3a-d
	$$
	which implies that
	$$
	  a^2-3a+2-2s=0.
	$$
	The discriminant $\Delta$ of this equation and the solutions are
	$$
	  \Delta = 1 + 8s, \ \ \  a=\frac{3 \pm \sqrt{\Delta}}{2}.
	$$
	$\Delta$ is a square if and only if $s=0,1,3,6$ and, respectively, we obtain $\sqrt{\Delta} = 1,3,5,7$.
	
	If $s=0$, then $a=1,2$ and we obtain the divisors
	$$
	  \begin{array}{lll}
	  	D=l-e_1 - e_2, & \text{ for } r = 2, & d = 1, \\
		D=2l-e_1 - \cdots - e_r, & \text{ for } r \leq 5, & d = 6-r.
	  \end{array}
	$$
	
	If $s=1$, then $a=3$ and we obtain, for any $r \geq 2$, the divisors
	$$
	  D = 3l - 2e_1 - e_2 - \cdots - e_r, \ \ \ \ d = 8 - r.
	$$
	
	If $s=3$, then $a=4$ and we obtain, for $r \geq 3$, the divisors
	$$
	  D = 4l - 2e_1 - 2e_2 - 2e_3 - e_4 - \cdots - e_r, \ \ \ \ d = 9 - r.
	$$
	
	If $s=6$, then $a=5$ and we obtain, for $r=6$, the divisor
	$$
	  D = 5l - 2e_1 - 2e_2 - 2e_3 - 2e_4 - 2e_5 - 2e_6, \ \ \ \ d = 3.
	$$
	
	Since, by construction, any divisor $D$ above satisfies $D^2 = D.H - 2$ and $d=D.H > 0$, we can conclude that they are all ACM initialized line bundles.
\end{proof}

\subsection{Bounds for the intersection product of ACM divisors}

We want to end this section giving a bound for the intersection product of two ACM, initialized divisors. Let's start with the upper bound. We need the following lemma:

\begin{lemma}
	Let $X\subseteq\PP^n_k$ be a del Pezzo surface of degree $n$ and let $C, D$ be smooth rational curves on $X$ of respective degree $c, d$. Then, if $(m-1)n < c+d \leq mn$, where $m \geq 1$, we have
	\begin{align}\label{general bound for intersection product of rational curves}
		  C.D \leq 2 + (m-1)(c + d) - m(m-1)n/2
	\end{align}
	and the equality occurs if and only if $C+D\sim mH$. In this case, if $ m > 1$, we also have $c=d=mn/2$, so that $mn$ has to be even.
\end{lemma}
\begin{proof}
	Set $E:=C+D - mH$. We will use the fact that $C^2=c-2$ and $D^2=d-2$.
	
	If $E$ is effective, since $E.H \leq 0$, we obtain $C+D\sim mH$ and therefore, multiplying $mH$ by $C$ and $D$,
	$$
	C.D = (m-1)d + 2 = (m-1)c + 2.
	$$
	If $m=1$ the bound in (\ref{general bound for intersection product of rational curves}) is reached. This also happens if $m > 1$, since in this case we have $c=d=mn/2$.
	
	Assume now that $E$ is not effective, i.e. $\h^0(E) = 0$. We have
	$$
	(m-1)n - c -d < 0 \then  \h^2(E)=\h^0((m-1)H - C -D)=0.
	$$
	So Riemann-Roch gives
	$$
	  -\h^1(E) = \chi(E) = C.D - 1 - (m-1)(c+d) + m(m-1)n/2 \leq 0
	$$
	and so the desired formula.
\end{proof}

By Theorem \ref{maintheorem} we can conclude that:
\begin{cor}\label{upperbound}
	Let $X\subseteq\PP^n_k$ be a del Pezzo surface of degree $n$ and let $C, D$ be non zero ACM, initialized divisors of respective degree $c, d$. Then:
	\begin{enumerate}
	\item if $c+d > n$ we have
	$$
	 C.D \leq 2 + c + d - n
	$$
	and the equality occurs if and only if $C+D\sim 2H$.
	In this case $c=d=n$.
		
	\item if $c+d \leq n$ we have
	$$
	 C.D \leq 2
	$$
	and the equality occurs if and only if $C+D\sim H$.
	In this case $c+d=n$.
	\end{enumerate}
\end{cor}

Regarding the lower bound, we have the following:
\begin{lemma}\label{lowerbound}
	Let $X\subseteq\PP^n_k$ be a del Pezzo surface of degree $n$ and let $C, D$ be non zero ACM divisors on $X$ of respective degree $c, d$. Then
	\begin{align}
		  C.D \geq -2 + min \{c,d \}
	\end{align}
	and the equality occurs if and only if $D\sim C$.
\end{lemma}

\begin{proof}
The proof is analogous to the previous lemma: by Theorem \ref{maintheorem} we know that $C$ and $D$ correspond to rational normal curves with $1\leq c,d\leq n$; let's suppose that $d\leq c$ and let's define $E:=D-C$. If $E\sim 0$ then multiplying by $D$ we get $D.C=d-2$. Otherwise we have
$$
\h^0(E)=0 \text{ and } \h^2(E)=\h^0(C-D-H)=0
$$
\noindent and therefore, by Riemann-Roch:
$$
0\geq -\h^1(E)=\chi(E)=\frac{(D-C)(D-C+H)}{2}+1=-1+d-C.D.
$$

\end{proof}

\section{ACM bundles of higher rank}

In the last section our aim is to construct ACM bundles of rank $n$ for any $n\geq 2$. In particular, we're going to see that strong del Pezzo surfaces of degree $\leq 6$ are of wild representation type. Notice that the bundles $\shE$ that we're going to obtain are simple, i.e, $\Hom(\shE,\shE)\cong k$ and, therefore, they are indecomposable. So they represent new ACM bundles that don't come from direct sums of the known ACM line bundles.
\subsection{Extensions of bundles} We remark the following property at the beginning since it will be very useful in this section.

\begin{prop}\label{ACM divisors of maximal degree are 0 regular}
	Let $X$ be a strong del Pezzo surface of degree $n$ and let $D$ be an initialized, ACM divisor on $X$. Then $\odi{X}(D)$ is $0$-regular if and only if $\deg D = D.H = n$.
\end{prop}

\begin{proof}
	Write $d$ for the degree of $D$ and set $\shL = \odi{X}(D)$. If $d=0$, i.e. $\shL \cong \odi{X}$, since
	$$
	\h^2(\odi{X}(-2))=\chi(-2H) = n +1 \neq 0
	$$
	$\shL$ is not $0$-regular. So we can assume $d > 0$.
	Since $\h^0(\shL(-2)) = \h^1(\shL(-2))=0$, using Theorem \ref{maintheorem}, we have
	$$
	      \h^2(\shL(-2)) = \chi(D-2H) = n - D.H.
	$$
	Finally, being $\h^1(\shL(-1)) = 0$ by hypothesis, we can conclude that $\shL$ is $0$-regular if and only if $d=D.H=n$, as required.
\end{proof}

Given a projective variety $X$ and coherent sheaves $\shF, \shG$ on it, we're going to be interested in extensions of the form

$$
	0 \arr \shF \arr \shE \arr \shG \arr 0.
$$
Given another extension
$$
	0 \arr \shF \arr \shE' \arr \shG \arr 0
$$
we are going to say that they are \emph{equivalent} if there exists an isomorphism $\psi:\shE\longrightarrow\shE'$ such that the following diagram commutes:

$$
\xymatrix{
0 \ar[r] &  \shF \ar[r] \ar@{=}[d] & \shE \ar[r] \ar[d]^\psi & \shG \ar[r] \ar@{=}[d] & 0\\
0 \ar[r] &  \shF \ar[r] & \shE' \ar[r] & \shG \ar[r] & 0.\\
}
$$

A \emph{weak equivalence} of extensions is similarly defined, except that we don't require the morphisms $\shF\longrightarrow\shF$ and $\shG\longrightarrow\shG$ to be the identity but only isomorphisms.

It's a well-known result that equivalent classes of extensions of $\shG$  by $\shF$ correspond bijectively to the elements of $\Ext^1(\shG,\shF)$. If $$
	0 \arr \shF \arr \shE \arr \shG \arr 0
$$

\noindent is such an extension, the corresponding element $[\shE] \in \Ext^1(\shG,\shF)$ is the image of $\id_{\shF}$ under the morphism
	$$
	  \Hom(\shF,\shF) \arrdi{\delta} \Ext^1(\shG,\shF)
	$$
	obtained applying $\Hom(-,\shF)$ to the exact sequence above. We will use the symbol $\delta$ for this morphism. The \emph{trivial} extension $\shF\oplus \shG$ corresponds to $0\in\Ext^1(\shG,\shF)$. Inside $\Ext^1(\shG,\shF)$ weak equivalence defines an equivalent relation that will be denoted by $\thicksim_w$.


\begin{dfn}
Given a variety $X$, a coherent sheaf $\shE$ on it is called \emph{simple} if $\Hom(\shE,\shE)\cong k$.
\end{dfn}

\begin{prop}\label{base step for extensions}
	Let $X$ be a projective variety over $k$ and $\shF_1, \dots,\shF_{r+1}$, with $r \geq 1$, be simple coherent sheaves on $X$ such that
	$$
	  \Hom(\shF_i,\shF_j)=0 \text{ for } i \neq j.
	$$
	Denote also
	$$
	  U=\Ext^1(\shF_{r+1},\shF_1)-\{0\} \times \cdots \times \Ext^1(\shF_{r+1},\shF_r)-\{0\} \subseteq \Ext^1(\shF_{r+1},\bigoplus_{i=1}^r \shF_i).
	$$
	Then a sheaf $\shE$ that comes up from an extension of $\shF_{r+1}$ by $\bigoplus_{i=1}^r \shF_i$ is simple if and only if $[\shE] \in U$ and given two extensions $[\shE], [\shE'] \in U$ we have that
	$$
	  \Hom(\shE,\shE') \neq 0 \iff [\shE]\thicksim_w [\shE'].
	$$
	To be more precise, the simple coherent sheaves $\shE$ coming up from an extension of $\shF_{r+1}$ by $\bigoplus_{i=1}^r \shF_i$
	$$
	0 \arr \bigoplus_{i=1}^r \shF_i \arr \shE \arr \shF_{r+1} \arr 0
	$$
	are parametrized, up to isomorphisms (of coherent sheaves), by
	$$
	  (U/\thicksim_w) \cong\PP(\Ext^1(\shF_{r+1},\shF_1)) \times \cdots \times \PP(\Ext^1(\shF_{r+1},\shF_r)).
	$$
\end{prop}

\begin{proof}
	Set $\shF:= \bigoplus_{i=1}^r \shF_i$ and $V:=\bigoplus_{i=1}^r \Ext^1(\shF_{r+1},\shF_i) \cong \Ext^1(\shF_{r+1},\shF)$. Note that under the isomorphisms
  \[
  \begin{tikzpicture}[xscale=2.8,yscale=-0.9]
    \node (A0_0) at (0, 0) {$\Hom(\shF,\shF)$};
    \node (A0_1) at (1, 0) {$\displaystyle{\bigoplus_{i=1}^r \Hom(\shF_i,\shF_i)}$};
    \node (A0_2) at (2, 0) {$k^r$};
    \node (A1_0) at (0, 1) {$\id_\shF$};
    \node (A1_1) at (1, 1) {$(\id_{\shF_1},\dots,\id_{\shF_r})$};
    \node (A1_2) at (2, 1) {$(1,\dots,1)$};
    \path (A0_0) edge [->] node [auto] {$\scriptstyle{\cong}$} (A0_1);
    \path (A1_0) edge [|->] node [auto] {$\scriptstyle{}$} (A1_1);
    \path (A0_1) edge [->] node [auto] {$\scriptstyle{\cong}$} (A0_2);
    \path (A1_1) edge [|->] node [auto] {$\scriptstyle{}$} (A1_2);
  \end{tikzpicture}
  \]
we can identify $\id_{\shF_i}$ with elements of $\Hom(\shF,\shF)$ such that $\id_{\shF}=\sum_{i=1}^r\id_{\shF_i}$. Moreover we have	
	$$
	  \Hom(\shF,\shF_{r+1}) = \Hom(\shF_{r+1},\shF) = 0.
	$$
	
	\emph{First Step.} Let $[\shE] = (\eta_1,\dots,\eta_r) \in V$ and take the corresponding exact sequence
	\begin{align}\label{exact sequence for the global extension}
	 0 \arr \shF \arrdi{\alpha} \shE \arrdi{\beta} \shF_{r+1} \arr 0.
	\end{align}
	We claim that the map
	$$
	  k^r \cong \Hom(\shF,\shF) \arrdi{\delta} V
	$$
	verifies that $\delta(\id_{\shF_i})=\eta_i$ for all $i$. In particular $\delta$ is injective if and only if $\eta_i\neq 0$ for all $i=1,\dots,r$ and therefore if and only if $[\shE] \in U$.
	Indeed, applying $\Hom(-,\shF)$ to (\ref{exact sequence for the global extension}) and using the isomorphism $\Hom(-,\shF) \cong \bigoplus_{i=1}^r \Hom(-,\shF_i)$ and the fact that one can take as an injective resolution of $\shF$ the direct sum of injective resolutions of $\shF_i$, we have a commutative diagram (see \cite[Chapter III, Proposition 6.4]{Harti})
  \[
  \begin{tikzpicture}[xscale=2.9,yscale=-0.6]
    \node (A1_0) at (0, 1) {$\Hom(\shF_i,\shF_i)$};
    \node (A1_1) at (1, 1) {$\Hom(\shF,\shF_i)$};
    \node (A1_2) at (2, 1) {$\Ext^1(\shF_{r+1},\shF_i)$};
    \node (A3_1) at (1, 3) {$\Hom(\shF,\shF)$};
    \node (A3_2) at (2, 3) {$\Ext^1(\shF_{r+1},\shF)$};
    \path (A3_1) edge [->] node [auto] {$\scriptstyle{\delta}$} (A3_2);
    \path (A1_2) edge [->] node [auto] {$\scriptstyle{}$} (A3_2);
    \path (A1_1) edge [->] node [auto] {$\scriptstyle{}$} (A1_2);
    \path (A1_1) edge [->] node [auto] {$\scriptstyle{}$} (A3_1);
    \path (A1_0) edge [->] node [auto] {$\scriptstyle{\cong}$} (A1_1);
  \end{tikzpicture}
  \]
      \noindent which tells us that $\delta$ verifies $\delta(\id_{\shF_i})\in\Ext^1(\shF_{r+1},\shF_i)$. Finally, by linearity, we have that
      $$
      (\eta_1, \dots, \eta_r)=\delta(\id_{\shF})=\sum_{i=1}^r\delta(\id_{\shF_i})=(\delta(\id_{\shF_1}),\dots,\delta(\id_{\shF_r})).
      $$

      \emph{Second step.} We claim that $\shE$ is simple if and only if $[\shE] \in U$.
      Applying $\Hom(-,\shF)$ to (\ref{exact sequence for the global extension}) we have an exact sequence
      $$
	 0 \arr \Hom(\shE,\shF) \arr \Hom(\shF,\shF) \arrdi{\delta} \Ext^1(\shF_{r+1},\shF)
      $$
      and therefore $\Hom(\shE,\shF) = 0$ if and only if $\delta$ is injective and so if and only if $[\shE] \in U$. Applying now $\Hom(-,\shF_{r+1})$ to the same sequence we get
      $$
	0 \arr \Hom(\shF_{r+1},\shF_{r+1}) \arr \Hom(\shE,\shF_{r+1}) \arr 0
      $$
      which tells us that $\Hom(\shE,\shF_{r+1}) \cong k$ is generated by $\beta$ (see (\ref{exact sequence for the global extension})). Finally we apply $\Hom(\shE,-)$ again to (\ref{exact sequence for the global extension}) obtaining the exact sequence
      $$
      0 \arr \Hom(\shE,\shF) \arr \Hom(\shE,\shE) \arr \Hom(\shE,\shF_{r+1}) \cong k \arr 0
      $$
      where the surjectivity of the second map follows from the fact that $\id_{\shE}$ is sent to $\beta$. So we can conclude that $\shE$ is simple if and only if $\Hom(\shE,\shF) = 0$ and so if and only if $[\shE] \in U$.

      \emph{Third step.} We're going to prove the following claim: let $[\shE]=(\eta_1,\dots,\eta_r)$, $[\shE']=(\xi_1,\dots,\xi_r)$ be extensions from $U$, the first one corresponding to the sequence (\ref{exact sequence for the global extension}), the second one to
      $$
	  0 \arr \shF \arrdi{\alpha'} \shE' \arrdi{\beta'} \shF_{r+1} \arr 0.
      $$
      Then
      $$
	  \Hom(\shE,\shE') \neq 0  \iff \forall i \ \ \ \exists \omega_i \in k^* \ \ \text{ s.t. } \ \ \xi_i = \omega_i \eta_i,
      $$
      and in this case $[\shE] \thicksim_w [\shE']$.

      $\Longleftarrow)$ It's enough to check that, if $\shF \arrdi{\psi} \shF$ is the isomorphism given by a diagonal matrix with diagonal $(\omega_1, \dots, \omega_r)$, then the exact sequence on the first row of
  $$
  \begin{tikzpicture}[xscale=1.6,yscale=-1.1]
    \node (A0_0) at (0, 0) {$0$};
    \node (A0_1) at (1, 0) {$\shF$};
    \node (A0_2) at (2, 0) {$\shE'$};
    \node (A0_3) at (3, 0) {$\shF_{r+1}$};
    \node (A0_4) at (4, 0) {$0$};
    \node (A1_0) at (0, 1) {$0$};
    \node (A1_1) at (1, 1) {$\shF$};
    \node (A1_2) at (2, 1) {$\shE'$};
    \node (A1_3) at (3, 1) {$\shF_{r+1}$};
    \node (A1_4) at (4, 1) {$0$};
    \path (A0_1) edge [->] node [auto] {$\scriptstyle{\psi}$} (A1_1);
    \path (A0_0) edge [->] node [auto] {$\scriptstyle{}$} (A0_1);
    \path (A0_1) edge [->] node [auto] {$\scriptstyle{\alpha' \psi}$} (A0_2);
    \path (A1_0) edge [->] node [auto] {$\scriptstyle{}$} (A1_1);
    \path (A0_3) edge [->] node [auto] {$\scriptstyle{\id}$} (A1_3);
    \path (A1_1) edge [->] node [auto] {$\scriptstyle{\alpha'}$} (A1_2);
    \path (A0_3) edge [->] node [auto] {$\scriptstyle{}$} (A0_4);
    \path (A0_2) edge [->] node [auto] {$\scriptstyle{\id}$} (A1_2);
    \path (A1_2) edge [->] node [auto] {$\scriptstyle{\beta'}$} (A1_3);
    \path (A0_2) edge [->] node [auto] {$\scriptstyle{\beta'}$} (A0_3);
    \path (A1_3) edge [->] node [auto] {$\scriptstyle{}$} (A1_4);
  \end{tikzpicture}
  $$
\noindent corresponds to $[\shE]$. In fact in this case we will obtain a weak equivalence $[\shE] \thicksim_w [\shE']$. So let's
apply the functor $\Hom(-,\shF)$ to the above diagram; using the properties of the derived functors, we get a commutative diagram
  \[
  \begin{tikzpicture}[xscale=3.2,yscale=-1.1]
    \node (A0_0) at (0, 0) {$\Hom(\shF,\shF)$};
    \node (A0_1) at (1, 0) {$\Ext^1(\shF_{r+1},\shF)$};
    \node (A1_0) at (0, 1) {$\Hom(\shF,\shF)$};
    \node (A1_1) at (1, 1) {$\Ext^1(\shF_{r+1},\shF)$};
    \path (A0_0) edge [->] node [auto] {$\scriptstyle{\delta'}$} (A0_1);
    \path (A0_0) edge [->] node [auto,swap] {$\scriptstyle{\psi^*}$} (A1_0);
    \path (A0_1) edge [->] node [auto] {$\scriptstyle{\id}$} (A1_1);
    \path (A1_0) edge [->] node [auto,swap] {$\scriptstyle{\delta''}$} (A1_1);
  \end{tikzpicture}
  \]

\noindent where $\delta'$ (respectively $\delta''$) is the connecting morphism corresponding to the exact sequence on the second row (resp. on the first row).
Using the usual identification $\Hom(\shF,\shF) \cong k^r$, the map $\psi^*$ has the same representation of $\psi$ as matrix, i.e. it is diagonal with entries $(\omega_1,\dots,\omega_r)$. So we get the relation
      $$
	\delta''(\id_\shF)=\delta'(\psi^{*-1}(1,\dots,1))=\delta'(\omega_1^{-1},\dots, \omega_r^{-1}) = (\omega_1^{-1}\xi_1,\dots,\omega_r^{-1}\xi_r)=\delta(\id_\shF).
      $$

      $\then)$ Let $\shE \arrdi{u} \shE'$ be a non zero map. We start defining the dashed arrows in
  $$
  \begin{tikzpicture}[xscale=1.6,yscale=-1.1]
    \node (A0_0) at (0, 0) {$0$};
    \node (A0_1) at (1, 0) {$\shF$};
    \node (A0_2) at (2, 0) {$\shE$};
    \node (A0_3) at (3, 0) {$\shF_{r+1}$};
    \node (A0_4) at (4, 0) {$0$};
    \node (A1_0) at (0, 1) {$0$};
    \node (A1_1) at (1, 1) {$\shF$};
    \node (A1_2) at (2, 1) {$\shE'$};
    \node (A1_3) at (3, 1) {$\shF_{r+1}$};
    \node (A1_4) at (4, 1) {$0$};
    \path (A0_1) edge [->,dashed] node [auto] {$\scriptstyle{\psi}$} (A1_1);
    \path (A0_0) edge [->] node [auto] {$\scriptstyle{}$} (A0_1);
    \path (A0_1) edge [->] node [auto] {$\scriptstyle{\alpha}$} (A0_2);
    \path (A1_0) edge [->] node [auto] {$\scriptstyle{}$} (A1_1);
    \path (A0_3) edge [->,dashed] node [auto] {$\scriptstyle{\lambda}$} (A1_3);
    \path (A1_1) edge [->] node [auto] {$\scriptstyle{\alpha'}$} (A1_2);
    \path (A0_3) edge [->] node [auto] {$\scriptstyle{}$} (A0_4);
    \path (A0_2) edge [->] node [auto] {$\scriptstyle{u}$} (A1_2);
    \path (A1_2) edge [->] node [auto] {$\scriptstyle{\beta'}$} (A1_3);
    \path (A0_2) edge [->] node [auto] {$\scriptstyle{\beta}$} (A0_3);
    \path (A1_3) edge [->] node [auto] {$\scriptstyle{}$} (A1_4);
  \end{tikzpicture}.
  $$
	Since $\beta' u \in \Hom(\shE,\shF_{r+1})=<\beta>_k$ there exists $\lambda \in k$ such that $\beta' u = \lambda \beta$. Since $\lambda$ is defined, also $\psi$ is automatically defined. If $\lambda = 0$, then $\beta' u = 0$ and so $u$ factorizes through a map $\shE \arr \shF$. Since we have proved that $\Hom(\shE,\shF)=0$ we also obtain $u=0$, a contradiction. So we get $\lambda \in k^*$.
	
	We can note that, since $\Hom(\shF_i,\shF_j)=0$ for $i \neq j$ and all the $\shF_i$ are simple, $\psi$ has to be a diagonal matrix. Call $(\mu_1,\dots,\mu_r)$ its diagonal. Applying now $\Hom(-,\shF)$ to the above diagram we have a commutative diagram
  \[
  \begin{tikzpicture}[xscale=3.2,yscale=-1.1]
    \node (A0_0) at (0, 0) {$\Hom(\shF,\shF)$};
    \node (A0_1) at (1, 0) {$\Ext^1(\shF_{r+1},\shF)$};
    \node (A1_0) at (0, 1) {$\Hom(\shF,\shF)$};
    \node (A1_1) at (1, 1) {$\Ext^1(\shF_{r+1},\shF)$};
    \path (A0_0) edge [->] node [auto] {$\scriptstyle{\delta'}$} (A0_1);
    \path (A0_0) edge [->] node [auto,swap] {$\scriptstyle{\psi^*}$} (A1_0);
    \path (A0_1) edge [->] node [auto] {$\scriptstyle{\lambda}$} (A1_1);
    \path (A1_0) edge [->] node [auto,swap] {$\scriptstyle{\delta}$} (A1_1);
  \end{tikzpicture}
  \]
	and, as remarked above, $\psi^*$ is a diagonal matrix with entries $(\mu_1,\dots,\mu_r)$. The commutativity of the above diagram gives relations
	$$
	  \lambda \xi_i = \mu_i \eta_i.
	$$
	Since we are assuming that $\eta_i,\xi_i \neq 0$, we obtain $\mu_i \in k^*$ and the $\omega_i = \mu_i/\lambda$ satisfy the requirements.
\end{proof}

\begin{rem}\label{Ext i for ACM divisors of maximal degree}
	Let $X$ be a strong del Pezzo surface of degree $d$ and let $C, D$ be distinct initialized ACM divisors of maximal degree $d$. Then $C-D$ and $D-C$ can not be equivalent to an effective divisor since otherwise, having $C.H=D.H=d$, we must have $C \sim D$. So
	$$
	  \h^0(C-D)=\h^2(C-D)=0 \then -\h^1(C-D)=\chi(C-D)=d-C.D-1.
	$$
	We can conclude that
	$$
	  \Ext^2(\odi{X}(C),\odi{X}(D)) = \Hom(\odi{X}(C),\odi{X}(D))=0
	$$
	and
	$$
	  \dim_k \Ext^1(\odi{X}(C),\odi{X}(D)) = 1+C.D-d.
	$$
\end{rem}


\begin{thm}\label{thetheorem}
	Let $X\subseteq\PP^d$ be a strong del Pezzo surface of degree $d$ less or equal than six. Then for any integer $n \geq 2$ there exists a family of dimension $\geq n-1$ of non-isomorphic initialized simple $0$-regular ACM vector bundles of rank $n$.
\end{thm}
\begin{proof}
	We know that $X$ corresponds to the blow-up of $r=9-d$ points in general position of $\PP^2_k$. Let $C,D$ be distinct initialized ACM divisors of maximal degree $d$ satisfying the condition
	$$
	C.D = 1 + d.
	$$
	Before continuing we show that we can always find a couple $(C,D)$ satisfying the above condition. Below by $X^r$ we mean a del Pezzo surface blow-up of $r$ points of $\PP^2_k$.
	\begin{center}
	\begin{tabular}{|c|c|l|c|}
	\hline
		$X$   & $d$&$
			    \begin{array}{l}
		            	C \\ D
		            \end{array}
			    $
			   & $C.D=1+d$ \\
	\hline
		$X^3$ & 6 & $
			      \begin{array}{l}
			      	3l-2e_1-e_2 \\ 3l-2e_2 - e_3
			      \end{array}
			    $
			  & $7$ \\
	\hline
		$X^4$ & 5 & $
			      \begin{array}{l}
			      	3l-2e_1-e_2 - e_3 \\ 3l-2e_2-e_3 - e_4
			      \end{array}
			    $
			  & $6$ \\
	\hline
		$X^5$ & 4 & $
			      \begin{array}{l}
			      	3l-2e_1-e_2 - e_3 -e_4\\ 3l-2e_2-e_3 - e_4 - e_5
			      \end{array}
			    $
			  & $5$ \\
	\hline
		$X^6$ & 3 & $
			      \begin{array}{l}
			      	3l-2e_1-e_2 - e_3 -e_4-e_5\\ 3l-2e_2-e_3 - e_4 - e_5-e_6
			      \end{array}
			    $
			  & $4$ \\
	\hline
	\end{tabular}
	\end{center}
	
	Set $E=2H-C$ and $F=2H-D$. Since $E^2 = E.H -2$ and $E.H=d$ and the same is true for $F$, we have that $E,F$ also are initialized ACM divisors of maximal degree. A direct computation gives the equalities:
	\begin{align*}
		1+C.E -d= 1+D.F -d= 3 \\
		1+C.D-d = 1 + E.F -d = 2 \\
		1+D.E-d = 1+C.F -d = 0.
	\end{align*}
	It's clear from Lemma \ref{lowerbound} that $C,D,E,F$ are distinct as equivalence classes. In what follows we make use of Proposition \ref{base step for extensions} and Remark \ref{Ext i for ACM divisors of maximal degree}. Moreover, from Proposition \ref{ACM divisors of maximal degree are 0 regular}, we have that the invertible sheaves associated to the divisors $C,D,E,F$ are $0$-regular. Since all the vector bundles obtained below are subsequent extensions of ACM, initialized and $0$-regular line bundles, the same condition will be satisfied by those bundles.
	
	\emph{Rank 2}. It's enough to take extensions of $\odi{X}(C)$ by $\odi{X}(E)$, which satisfy the hypothesis of Proposition \ref{base step for extensions}. In this way we obtain a family parametrized by $\PP^2$ of simple ACM vector bundles.
	
	\emph{Rank $2m+1$.} First consider extensions of $\odi{X}(D)$ by $\odi{X}(C)$. Again by \ref{base step for extensions} this gives a family parametrized by $\PP^1$ of simple vector bundles without non-zero morphisms among them. So we can take distinct elements $\shE_1, \dots, \shE_m \in \PP^1$, i.e. satisfying the exact sequences
	\begin{align}\label{first exact sequence for extensions}
	  0 \arr \odi{X}(C) \arr \shE_i \arr \odi{X}(D) \arr 0.
	\end{align}
	Now let's consider extensions of the form
	\begin{align}\label{second exact sequence for extensions}
	  0 \arr \bigoplus_{i=1}^m \shE_i \arr \shH \arr \odi{X}(E) \arr 0.
	\end{align}
	Applying $\Hom(-,\odi{X}(E))$ and $\Hom(\odi{X}(E),-)$ to (\ref{first exact sequence for extensions}) we see that
	$$
	  \Hom(\shE_i,\odi{X}(E)) = \Hom(\odi{X}(E),\shE_i) = 0.
	$$
	Therefore we can deduce that the sheaves $\shE_1,\dots,\shE_m,\odi{X}(E)$ satisfy the hypothesis of  Proposition \ref{base step for extensions}. We have to compute $\Ext^1(\odi{X}(E),\shE_i)$. In general, if $R$ is any ACM, initialized divisors of maximal degree $d$, applying $\Hom(\odi{X}(R),-)$ to (\ref{first exact sequence for extensions}), we obtain the exact sequences
	$$
	0 \arr \Ext^1(\odi{X}(R),\odi{X}(C)) \arr \Ext^1(\odi{X}(R),\shE_i) \arr \Ext^1(\odi{X}(R),\odi{X}(D)) \arr 0
	$$
	and so
	\begin{align}\label{dimension of ext 1 for R E i}
	  \dim_k \Ext^1(\odi{X}(R),\shE_i) = 2 - 2N + C.R + D.R.
	\end{align}
	If we take $R=E$ we find $\dim_k \Ext^1(\odi{X}(E),\shE_i) = 3$. So we have a family parametrized by $(\PP^2)^m$ of simple vector bundles of rank $2m+1$.
	
	\emph{Rank $2m+2$.} Let $\shH$ be one of the extensions of rank $2m+1$ obtained above. We consider the extensions of the form
	$$
	  0 \arr \shH \arr \shM \arr \odi{X}(F) \arr 0.
	$$
	Applying $\Hom(\odi{X}(F),-)$ and $\Hom(-,\odi{X}(F))$ to both (\ref{first exact sequence for extensions}) and (\ref{second exact sequence for extensions}) we get
	$$
	  \Hom(\shH,\odi{X}(F)) = \Hom(\odi{X}(F),\shH) = 0.
	$$
	So we are again in the hypothesis of Proposition \ref{base step for extensions}. We have to compute $\Ext^1(\odi{X}(F),\shH)$. If we apply $\Hom(\odi{X}(F),-)$ to (\ref{second exact sequence for extensions}) we get an exact sequence
	$$
	  0 \arr \Ext^1(\odi{X}(F),\bigoplus_{i=1}^m \shE_i) \arr \Ext^1(\odi{X}(F),\shH) \arr \Ext^1(\odi{X}(F),\odi{X}(E)) \arr 0,
	$$
	where the vanishing of $\Ext^2(\odi{X}(F),\shE_i)$ follows applying $\Hom(\odi{X}(F),-)$ to (\ref{first exact sequence for extensions}) and remembering that
	$$
	  \Ext^2(\odi{X}(F),\odi{X}(C)) = \Ext^2(\odi{X}(F),\odi{X}(D)) = 0.
	$$
	So, using also (\ref{dimension of ext 1 for R E i}), we have
	$$
	  \dim_k \Ext^1(\odi{X}(F),\shH) = 2 + 3m.
	$$
	In this way we obtain a family parametrized by $\PP^{1+3m}$ of simple ACM vector bundles of rank $2m+2$.
\end{proof}

We will end this paper showing that bundles constructed on the previous theorem are semistable and unstable. Following \cite{HuL}, we recall that a vector bundle $\shE$ on a smooth projective variety $X\subseteq\PP^d$ is \emph{semistable} if for every nonzero coherent subsheaf $\shF$ of $\shE$ we have the inequality
$$
P(\shF)/\rk(\shF)\leq P(\shE)/\rk(\shE),
$$
where $P(\shE)$ is the Hilbert polynomial of the sheaf and the order is with respect to their asymptotic behavior. If one has the strict inequality for any proper subsheaf the bundle is called \emph{stable}. There is another definition using the \emph{slope}, which is defined as $\mu(\shE):=\de(\shE)/\rk(\shE)$. We say that $\shE$ is $\mu$-\emph{(semi)stable} if for every subsheaf $\shF$ of $\shE$ with $0<\rk\shF <\rk\shE$, $\mu(\shF)<\mu(\shE)$ (resp. $\mu(\shF)\leq\mu(\shE))$. The four notions are related as follows:

$$
\mu-stable\Rightarrow stable \Rightarrow semistable \Rightarrow \mu-semistable.
$$

In order to show the semistability of bundles constructed in \ref{thetheorem}, we are going to use the following result (cfr. \cite[Lemma 1.4]{Mar}):

\begin{lemma}\label{maruyama}
Assume that
$$
0\longrightarrow\shE' \longrightarrow\shE \longrightarrow\shE''\longrightarrow 0
$$
is a short exact sequence of coherent sheaves with $\frac{P(\shE)}{\rk(\shE)}=\frac{P(\shE')}{\rk(\shE')}=\frac{P(\shE'')}{\rk(\shE'')}$. Then $\shE$ is semistable if and only if $\shE'$ and $\shE''$ are semistable.
\end{lemma}

\begin{prop}
Let $X\subseteq\PP^d$ be a strong del Pezzo surface of degree $d$ less or equal than six. Then the ACM  bundles constructed in the proof of \ref{thetheorem} are all strictly semistable bundles of constant slope $d$.
\end{prop}
\begin{proof}
Any line bundle is semistable; therefore the semistability of the vector bundles from \ref{thetheorem} is proved by induction just noticing that they verify the hypothesis of Lemma \ref{maruyama}. They can not be stable since the construction exhibits a subbundle contradicting the definition.
\end{proof}

\section{Acknowledgements.}

This research was developed during the P.R.A.G.MAT.I.C school held at the University of Catania (Italy, September 2009). The authors would like to thank the local organizers as well as the speakers for the wonderful atmosphere created during the school. They also acknowledge the support from the University of Catania.

The authors would also like to thank Daniele Faenzi and Giorgio Ottaviani for having proposed us this problem and for their constant support during the school.

\newcommand{\etalchar}[1]{$^{#1}$}
\providecommand{\bysame}{\leavevmode\hbox to3em{\hrulefill}\thinspace}
\providecommand{\MR}{\relax\ifhmode\unskip\space\fi MR }
\providecommand{\MRhref}[2]{%
  \href{http://www.ams.org/mathscinet-getitem?mr=#1}{#2}
}
\providecommand{\href}[2]{#2}


\begin{thebibliography}{KMMR{\etalchar{+}}01}

\bibitem[BGS87]{BGS}
Ragnar-Olaf Buchweitz, Gert-Martin Greuel, and Frank-Olaf Schreyer,
  \emph{{Cohen-Macaulay modules on hypersurface singularities II}}, Inventiones
  Mathematicae \textbf{88} (1987), no.~1, 165--182.

\bibitem[CH08]{CH}
Marta Casanellas and Robin Hartshorne, \emph{{ACM bundles on cubic surfaces}},
  Available from: \url{http://arxiv.org/abs/0801.3600}.

\bibitem[DG01]{DG}
Yuriy Drozd and Gert-Martin Greuel, \emph{{Tame and Wild Projective Curves and
  Classification of Vector Bundles}}, Journal of Algebra \textbf{246} (2001),
  no.~1, 1--54, Available from: \url{http://ns.imath.kiev.ua/\~{}drozd/vb.pdf}.

\bibitem[Dol09]{Dol}
Igor~V Dolgachev, \emph{{Topics in Classical Algebraic Geometry. Part I}},
  2009, Available from:
  \url{http://www.math.lsa.umich.edu/\~{}idolga/topics1.pdf}.

\bibitem[DPT80]{Dem}
Michel Demazure, Henry Pinkham, and Bernard Teissier, \emph{{S\'{e}minaire sur
  les singularit\'{e}s des surfaces. Surfaces de del Pezzo}}, Springer-Verlag,
  1980, Available from:
  \url{http://www.numdam.org/numdam-bin/browse?id=SSS\_1976-1977\_\_\_}.

\bibitem[EH88]{EH}
David Eisenbud and H.~J\"{u}rgen Herzog, \emph{{The classification of
  homogeneous Cohen-Macaulay rings of finite representation type}},
  Mathematische Annalen \textbf{280} (1988), no.~2, 347--352, Available from:
  \url{http://www.msri.org/\~{}de/papers/pdfs/1988-001.pdf}.

\bibitem[Eis02]{EisSyz}
David Eisenbud, \emph{{The Geometry of Syzygies: A Second Course in Commutative
  Algebra and Algebraic Geometry}}, 2002, Available from:
  \url{www.msri.org/people/staff/de/ready.pdf}.

\bibitem[Fae08]{Fa}
Daniele Faenzi, \emph{{Rank 2 arithmetically Cohen-Macaulay bundles on a
  nonsingular cubic surface}}, Journal of Algebra \textbf{319} (2008), no.~1,
  143--186, Available from: \url{http://arxiv.org/pdf/math/0504492v2}.

\bibitem[Har77]{Harti}
Robin Hartshorne, \emph{{Algebraic Geometry}}, Springer-Verlag, 1977.

\bibitem[HL97]{HuL}
Daniel Huybrechts and Manfred Lehn, \emph{{The geometry of moduli spaces of
  sheaves}}, 1997, Available from:
  \url{www.math.uni-bonn.de/people/huybrech/moduli.ps}.

\bibitem[KMMR{\etalchar{+}}01]{Mir}
Jan~Oddvar Kleppe, Juan Migliore, Rosa~Maria Mir\'{o}-Roig, Uwe Nagel, and
  Chris Peterson, \emph{{Gorenstein liaison, complete intersection liaison
  invariants and unobstructedness}}, AMS Bookstore, 2001.

\bibitem[Kn87]{Kno}
Horst Kn\"{o}rrer, \emph{{Cohen-Macaulay modules on hypersurface singularities
  I}}, Inventiones Mathematicae \textbf{88} (1987), no.~1, 153--164.

\bibitem[Kol96]{Kol}
J\'{a}nos Koll\'{a}r, \emph{{Rational curves on algebraic varieties}},
  Springer, 1996.

\bibitem[Man86]{Man}
Yuri Manin, \emph{{Cubic Forms}}, North Holland, 1986.

\bibitem[Mar78]{Mar}
Masaki Maruyama, \emph{{Moduli of stable sheaves II}}, J . Math. Kyoto Univ.
  \textbf{3} (1978), 557--614, Available from:
  \url{http://projecteuclid.org/DPubS/Repository/1.0/Disseminate?view=body\&id%
=pdf\_1\&handle=euclid.kjm/1250522511}.

\bibitem[Mig98]{Mig}
Juan~Carlos Migliore, \emph{{Introduction to liaison theory and deficiency
  modules}}, Birkh\"{a}user, 1998.

\end{thebibliography}
\end{document}